\documentclass[11pt,letterpaper]{amsart}
\usepackage[utf8]{inputenc}
\usepackage[T1]{fontenc}
\usepackage{amssymb}
\usepackage{graphicx}
\usepackage{url}
\usepackage{enumerate}
\usepackage{enumitem}
\usepackage{verbatim}
\usepackage{color}
\usepackage{tikz}
\usetikzlibrary{arrows.meta,bending}
\usetikzlibrary{decorations.pathmorphing}
\usetikzlibrary{decorations.markings}
\usepackage{array}
\usepackage{tikz-cd}
\usepackage{amsmath}
\usepackage{amsthm}
\usepackage[colorlinks=true,linkcolor=purple,citecolor=violet]{hyperref}
\usepackage{mathtools}
\usepackage{bbm}
\usepackage[capitalise, noabbrev]{cleveref}
\usepackage{fullpage}
\usepackage{lmodern}
\usepackage{mathrsfs}
\usepackage{esint}

\usetikzlibrary{positioning,arrows.meta}

\DeclareMathOperator{\bS}{\mathbb{S}}
\DeclareMathOperator{\bR}{\mathbb{R}}

\newcommand{\Ent}{\mathrm{Ent}}

\usepackage{lipsum}

\newtheorem{theorem}{Theorem}[section]
\newtheorem{lemma}[theorem]{Lemma}
\newtheorem{corollary}[theorem]{Corollary}
\newtheorem{proposition}[theorem]{Proposition}

\newtheorem{remark}[theorem]{Remark}

\theoremstyle{definition}

\title{Boundary Value Problems for the $SU(\infty)$-Toda Equation via the Onofri Inequality}
\author{Sun-Yung Alice Chang and Hongyi Liu}
\date{}

\newcommand{\Addresses}{%
	\bigskip
	\footnotesize
	
	\noindent
	\textsc{Department of Mathematics, Princeton University,
		Princeton, NJ 08544, USA.}\par\nopagebreak
	\noindent
	\textit{E-mail address}: 
	\href{mailto:syachang@math.princeton.edu}
	{\texttt{syachang@math.princeton.edu}}
	
	\medskip
	
	\noindent
	\textsc{Department of Mathematics, University of California, Irvine,
		Irvine, CA 92697, USA.}\par\nopagebreak
	\noindent
	\textit{E-mail address}: 
	\href{mailto:hongyi.liu@uci.edu}
	{\texttt{hongyi.liu@uci.edu}}
}

\begin{document}
	
	\begin{abstract}
		We formulate and solve a degenerate elliptic boundary value problem for the $SU(\infty)$-Toda
		equation on $[0,\frac12)\times\bS^2$. This constructs anti-self-dual Poincar\'e--Einstein metrics with circle symmetry on the $4$-ball and on complex line bundles over $\bS^2$.
	\end{abstract}

\maketitle

\tableofcontents

\section{Introduction}

We study boundary value problems for the $SU(\infty)$-Toda equation that arise in the construction of four-dimensional Einstein metrics.

Let $(\bS^2,g_{\bS^2})\subset (\bR^3,g_{\bR^3})$ be the standard round sphere with Gauss curvature 1, and let $d\mu$ be its normalized area measure, so that $\oint_{\bS^2}d\mu=1$. We write $\nabla$ and $\Delta=\mathrm{div}\nabla$ for the gradient and Laplacian in the metric $g_{\bS^2}$. The coordinate functions $x_i$ satisfy $\Delta x_i=-2x_i$, $i=1,2,3$. Let $x\mapsto-x$ denote the antipodal map on $\bS^2$.

We consider the boundary value problem for the $SU(\infty)$-Toda equation
\begin{align}\label{eq:intro_toda}
	\Delta w+(e^w)_{\xi\xi}=2,\quad (\xi,x)\in[0,\frac12)\times\bS^2.
\end{align}
The $SU(\infty)$ Toda equation has appeared extensively in four-dimensional differential geometry; see, for example, \cite{LeBrun1991Toda,Ward1990,Tod1997,Calderbank2000,CalderbankTod2001,Tod2006}. For a detailed account of the geometric reduction leading to \eqref{eq:intro_toda} in the normalization used here, see \cite[Sections 2.1 and 2.4]{LiLiu2025}.

We prescribe a function $\psi\in C^{2,\alpha}(\bS^2)$, $0<\alpha<1$, by
\begin{align}\label{eq:intro_toda_boundary}
	w(0,x)=\psi(x).
\end{align}

We assume that $w$ degenerates at $\xi=\frac12$ in the sense that
\begin{align}\label{eq:intro_degenerate_end}
	\lim\limits_{\xi\rightarrow\frac12}\oint_{\bS^2}e^{w(\xi,x)}\,d\mu=0.
\end{align}

Integrating \eqref{eq:intro_toda} over $\bS^2$ and using \eqref{eq:intro_degenerate_end} gives
\begin{align}\label{eq:intro_spherical_mean}
	\oint_{\bS^2}e^{w(\xi,x)}\,d\mu=(\frac12-\xi)(c-\xi).
\end{align}
Since $e^w>0$ for $0\leq\xi<\frac12$, we must have $c\geq\frac12$. In particular, the boundary data necessarily satisfy
\begin{align}\label{eq:intro_boundary_restriction}
	\oint_{\bS^2}e^\psi\,d\mu\geq\frac14.
\end{align}

\subsection{Main analytic results}

The necessary condition \eqref{eq:intro_boundary_restriction} turns out to be sufficient for the existence of a solution. Moreover, the equality and strict inequality cases determine its asymptotic behavior at the degenerate end. We prove the following existence, uniqueness, and asymptotic classification.

\begin{theorem}\label{thm:intro_asymptotics}
Every classical solution $w$ of \eqref{eq:intro_toda}--\eqref{eq:intro_toda_boundary} satisfying \eqref{eq:intro_degenerate_end} has one of the following two asymptotic behaviors as $\xi\rightarrow\frac12$, uniformly in $x\in\bS^2$:

Case 1:
\begin{align}\label{eq:intro_ball_end}
	\oint_{\bS^2}e^\psi\,d\mu=\frac14,\quad w=2\log(\frac12-\xi)+O(1).
\end{align}

Case 2:
\begin{align}\label{eq:intro_line_end}
	\oint_{\bS^2}e^\psi\,d\mu>\frac14,\quad w=\log(\frac12-\xi)+O(1).
\end{align}
Moreover, every boundary value $\psi\in C^{2,\alpha}(\bS^2)$ satisfying \eqref{eq:intro_boundary_restriction} admits exactly one such solution.
\end{theorem}

We prove Theorem \ref{thm:intro_asymptotics} by treating the two cases separately. Case 1 is more degenerate and constitutes the main analytic part of the paper. After a change of variables, the problem reduces to one on an infinite cylinder. The main issues are to establish  a priori $C^0$ estimates, prove existence for arbitrary boundary data satisfying the integral condition, and control the behavior at the degenerate end. An additional difficulty comes from the conformal invariance of the equation.

Indeed, if $\phi$ is a conformal transformation of $\bS^2$ and $J_\phi$ is its Jacobian with respect to $d\mu$, then
\begin{align}\label{eq:intro_conformal_action}
	w_\phi(\xi,x)=w(\xi,\phi(x))+\log J_\phi(x)
\end{align}
is again a solution of \eqref{eq:intro_toda}, with boundary value $\psi\circ\phi+\log J_\phi$. The following theorem establishes existence and a priori estimates and describes the asymptotic behavior.

\begin{theorem}[Case 1]\label{thm:intro_general_case}
	For any $\psi\in C^{2,\alpha}(\bS^2)$, $0<\alpha<1$, satisfying
	$$\oint_{\bS^2}e^\psi\,d\mu=\frac14,$$
	there is a unique classical solution $w$ of \eqref{eq:intro_toda}--\eqref{eq:intro_toda_boundary} satisfying \eqref{eq:intro_degenerate_end}. Moreover, there is a unique $a\in\bR^3$ such that, with
	$$\mathfrak{j}_a(x)=\frac{1}{(\sqrt{1+|a|^2}-a\cdot x)^2},$$
	\begin{align}\label{eq:intro_general_uniform_limit}
		\lim\limits_{\xi\rightarrow\frac12}\left(w(\xi,x)-2\log(\frac12-\xi)\right)=\log \mathfrak{j}_a(x),
	\end{align}
	uniformly in $x\in\bS^2$. We have $\|w\|_{C^{2,\alpha}([0,\frac13)\times\bS^2)}\leq C$, and for all nonnegative integers $j,m$, we have
	\begin{align}\label{eq:intro_general_limit}
		\left|\partial_\xi^j\nabla^m\left(w-2\log(\frac12-\xi)-\log \mathfrak{j}_a\right)\right|
		\leq C_{j,m}(\frac12-\xi)^{\delta-j},\quad  \forall\,\, \frac14 \leq \xi<\frac12. 
	\end{align}
	Here,
	$$\delta=\frac1{22},\quad C_{j,m}=C(j,m,\sup\psi,\inf\psi),\quad C=C(\|\psi\|_{C^{2,\alpha}(\bS^2)},\alpha,\sup\psi,\inf\psi).$$
\end{theorem}

The functions $\mathfrak{j}_a$ are the Jacobians of conformal transformations of $\bS^2$ and satisfy $$\Delta\log \mathfrak{j}_a+2\mathfrak{j}_a=2.$$

\begin{remark}\label{rmk:intro_even_case}
	It is natural to ask when the limit on the right-hand side of \eqref{eq:intro_general_uniform_limit} vanishes, or equivalently, when $a=0$. One sufficient condition is that the boundary data $\psi$ are even, that is, $\psi(-x)=\psi(x)$. By uniqueness, $w$ is also even, meaning $w(\xi,-x)=w(\xi,x)$, and hence $\mathfrak{j}_a$ is even, which forces $a=0$. On the other hand, the weaker balancing condition
	$
	\oint_{\bS^2}e^\psi x_i \,d\mu=0,\,  i=1,2,3,
	$
	is not sufficient to ensure that $a=0$.
\end{remark}

Our analysis of Case 1 begins with even boundary data. In this setting, an improved Onofri inequality yields entropy decay and the a priori estimates needed for the existence theory. We then extend the argument to general boundary data by replacing the entropy estimate with an estimate for a conformally invariant deficit involving the entropy and the $H^{-1}$ energy, analogous to the logarithmic Hardy--Littlewood--Sobolev deficit; see \cite{CarlenLoss1992,Beckner1993,Okikiolu2008,CarlenFigalli2013} and the references therein for related literature.

Case 2 is less degenerate and reduces, after a change of variables, to a problem on a bounded domain in higher dimension. This allows us to obtain the required a priori estimates and existence theory more directly.

\begin{theorem}[Case 2]\label{thm:intro_line_case}
	For any $\psi\in C^{2,\alpha}(\bS^2)$, $0<\alpha<1$, satisfying
	$$\oint_{\bS^2}e^\psi\,d\mu>\frac14,$$
	there is a unique classical solution $w$ of \eqref{eq:intro_toda}--\eqref{eq:intro_toda_boundary} satisfying \eqref{eq:intro_degenerate_end}. Moreover, 
	$$w(\xi,x)-\log(\frac12-\xi)$$
	extends to a function in $C^{2,\alpha}([0,\frac12]\times\bS^2)\cap C^\infty((0,\frac12]\times\bS^2)$.
	We have $\|w\|_{C^{2,\alpha}([0,\frac13)\times\bS^2)}\leq C$, and for all nonnegative integers $j,m$, we have
	\begin{align}\label{eq:intro_line_derivative_bound}
		\left|\partial_\xi^j\nabla^m\left(w-\log(\frac12-\xi)\right)\right|\leq C_{j,m},\quad \forall\,\,\frac14\leq\xi<\frac12.
	\end{align}
	Here, with $\epsilon=\oint_{\bS^2}e^\psi\,d\mu-\frac14>0$,
	$$C_{j,m}=C(j,m,\sup\psi,\inf\psi,\epsilon),\quad C=C(\|\psi\|_{C^{2,\alpha}(\bS^2)},\alpha,\sup\psi,\inf\psi,\epsilon).$$
\end{theorem}

\subsection{Geometric interpretation}

We next explain the geometric origin of \eqref{eq:intro_toda} and the consequences of the preceding analytic results.

LeBrun \cite{LeBrun1991Toda} showed that scalar-flat K\"ahler metrics in real dimension four admitting a Killing field are locally described by the $SU(\infty)$-Toda equation. Przanowski \cite{Przanowski1991} obtained local normal forms for anti-self-dual Einstein metrics with nonzero scalar curvature admitting a Killing field. Tod \cite{Tod2006} subsequently showed that an anti-self-dual Einstein metric with nonzero scalar curvature and such a symmetry is locally described by a solution of the $SU(\infty)$-Toda equation and is locally conformal to a scalar-flat K\"ahler metric; see also Dunajski--Tod \cite[Proposition 3.1]{DunajskiTod2010}. In the following, we use the normalization and global circle-bundle formulation of Li--Liu \cite[Section 2]{LiLiu2025}.

Given a solution $w$ of \eqref{eq:intro_toda}, set
\begin{align}\label{eq:intro_W}
	W=1-\frac{\xi}{2}w_\xi.
\end{align}
The LeBrun ansatz associates to $w$ a principal $\bS^1$-bundle over $(0,\frac12)\times\bS^2$, equipped with a connection one-form $\eta$ determined by $w$. On the total space of this bundle, define
\begin{align}\label{eq:intro_metric_ansatz}
	g=Wd\xi^2+W^{-1}\eta^2+We^wg_{\bS^2},
	\quad
	h=\xi^{-2}g.
\end{align}
Then $g$ is scalar-flat K\"ahler and $h$ is anti-self-dual Einstein, normalized by
\begin{align}\label{eq:intro_einstein}
	\mathrm{Ric}(h)=-3h.
\end{align}
For the solutions constructed above, Lemma \ref{lem:pe_W_positive} gives $W>0$, so both $g$ and $h$ are Riemannian metrics.

Assume now that the boundary data $\psi$ is smooth. The regularity of $w$ up to $\xi=0$ implies that the circle bundle and the metric $g$ extend smoothly to $\xi=0$. The added boundary is the circle bundle over $\bS^2$ defined by $\xi=0$. Since $W=1$ and $w=\psi$ there, the induced metric is
$$
g|_{\xi=0}=(\eta|_{\xi=0})^2+e^\psi g_{\bS^2}.
$$
Thus $h=\xi^{-2}g$ is conformally compact at $\xi=0$, with conformal infinity represented by $g|_{\xi=0}$.

At the other end, the derivative estimates in
Theorems \ref{thm:intro_general_case} and
\ref{thm:intro_line_case} imply that $W^{-1}\to0$
as $\xi\to\frac12$, so the circle fibers collapse.
The two cases lead to different smooth completions.

In Case 1, the circle bundle is the Hopf fibration
$$
\bS^1\hookrightarrow\bS^3\rightarrow\bS^2.
$$
With circle period $2\pi$, the level sets $\{\xi=\mathrm{const}\}\cong\bS^3$ collapse to a point as $\xi\to\frac12$. The asymptotics in Theorem \ref{thm:intro_general_case} imply that the metric extends smoothly across this point, which is an isolated fixed point of the $\bS^1$-action. The resulting manifold is diffeomorphic to $B^4$, and $h$ defines a complete Poincar\'e--Einstein metric on $B^4$. This is proved in Proposition \ref{prop:ball_metric_extension}; see also Li--Liu \cite[Section 3.3]{LiLiu2025} for the corresponding geometric model.

In Case 2, only the circle fibers collapse at $\xi=\frac12$, and the fixed point set is a copy of $\bS^2$. With circle period $\pi$, the metric extends smoothly across this fixed point set to the total space of the complex line bundle $\mathcal{O}(-m)\to\bS^2$, provided
\begin{align}\label{eq:intro_bundle_degree}
	\oint_{\bS^2}e^\psi\,d\mu=\frac{m-1}{4},
	\quad
	m\geq3 \text{ an integer}.
\end{align}
The condition \eqref{eq:intro_bundle_degree} expresses the compatibility between the boundary data and the bundle degree. The resulting metric $h$ is again complete and Poincar\'e--Einstein. The smooth extension is proved in Proposition \ref{prop:line_bundle_metric_extension}; see also \cite[Section 4.4]{LiLiu2025}.

\subsection{Relation to previous work}

The $B^4$ case is related to earlier work on self-dual Einstein fillings of $\bS^3$. LeBrun \cite{LeBrun1991Duke} formulated the positive frequency conjecture, which concerns the local structure, near the round conformal class on $\bS^3$, of conformal structures arising as conformal infinities of self-dual or anti-self-dual Einstein metrics on $B^4$. Biquard \cite{Biquard2002} subsequently proved this conjecture. In contrast, our construction is global within the $\bS^1$-invariant setting: the boundary data are not required to be a small perturbation of the round metric, and the nonlinear $SU(\infty)$-Toda equation produces complete anti-self-dual Poincar\'e--Einstein metrics on $B^4$ for the full class of boundary data considered here.

The line-bundle case is closely related to the work of Calderbank--Singer \cite{CalderbankSinger2004}, who constructed complete anti-self-dual Einstein metrics of negative scalar curvature on resolutions of cyclic quotient singularities, including the complex line bundles $\mathcal{O}(-m)\to\bS^2$ for $m\geq3$. Their construction assumes $\bS^1\times\bS^1$ symmetry. Under this stronger symmetry assumption, the Calderbank--Pedersen description \cite{CalderbankPedersen2002} reduces the anti-self-dual Einstein equation to a linear Laplace eigenvalue equation on the hyperbolic plane. In contrast, our construction assumes only $\bS^1$ symmetry and is governed by the nonlinear $SU(\infty)$-Toda equation.

The present work also builds on Li--Liu \cite{LiLiu2025}, where the Toda boundary value problem and the associated circle-bundle construction were studied over compact Riemann surfaces of positive genus. In contrast, the spherical case has no apparent maximum principle and admits an additional asymptotic regime leading to a smooth completion on $B^4$; this also differs from the cusp geometries arising over $T^2$ and higher-genus surfaces in Li--Liu \cite{LiLiuCusps2026}.

The paper is organized as follows. Section \ref{sec:comparison} establishes comparison estimates and uniqueness. Sections \ref{sec:even_estimates} and \ref{sec:even_existence} give the a priori estimates and the existence proof for even boundary data. Section \ref{sec:general_case} treats Case 1 without the evenness assumption. Section \ref{sec:line_bundle_case} treats Case 2. Finally, Section \ref{sec:pe_metrics} constructs the associated Poincar\'e--Einstein metrics and proves smooth extension across the fixed point sets of the circle action.

\subsection*{Acknowledgments}

The authors thank Lihe Wang for discussions on De Giorgi estimates and Paul Yang for suggesting that they first study the case of even boundary data. The second author also thanks Mingyang Li, Jingrui Cheng, and Peter Sarnak for useful conversations, and Song Sun for his encouragement. He is grateful to his wife, Jiasu Wang, for her unwavering support and constant encouragement throughout this work.

\section{Comparison estimates and uniqueness}\label{sec:comparison}

We first prove a comparison principle for \eqref{eq:intro_toda}. 

\begin{lemma}[Comparison principle]\label{lem:comparison}
	Let $w_1,w_2$ be smooth solutions of \eqref{eq:intro_toda} on $(a,b)\times\bS^2$. Assume
	\begin{align}\label{eq:comparison_boundary_order}
		\oint_{\bS^2}(e^{w_1}-e^{w_2})_+\,d\mu\rightarrow0
		\quad\text{as }\xi\rightarrow a\text{ and as }\xi\rightarrow b.
	\end{align}
	Then $w_1\leq w_2$ on $(a,b)\times\bS^2$. 
\end{lemma}

\begin{proof}
	Set $h=e^{w_1}-e^{w_2}$. By the mean value formula, $w_1-w_2=\Theta h$ for some positive smooth function $\Theta$ on $(a,b)\times\bS^2$. Note that for the arguments below, we do not require bounds on $\Theta$ near the endpoints. Subtracting the equations gives
	\begin{align}\label{eq:comparison_principle_difference_equation}
		L_\Theta h=h_{\xi\xi}+\Delta(\Theta h)=0.
	\end{align}
	We will show that
	\begin{align}
		L_\Theta h_+=\partial_{\xi\xi}h_++\Delta(\Theta h_+)\geq0
	\end{align}
	in the distributional sense: for any $0\leq\chi\in C_c^\infty((a,b)\times\bS^2)$, we have
	\begin{align}\label{eq:comparison_principle_distribution_explaination}
		\int_a^b\oint_{\bS^2}h_+(\chi_{\xi\xi}+\Theta\Delta\chi)\,d\mu\,d\xi\geq0.
	\end{align}
	This is a form of Kato's inequality; we include the proof for completeness.
	Let $\beta_{\epsilon}\in C^\infty(\bR)$ be a smooth convex approximation of $r_+=\max\{r,0\}$, satisfying
	\begin{align}
		\beta_{\epsilon}(r)=0 \, (r\leq0),\quad \beta_{\epsilon}'(r)=1 \, (r\geq\epsilon),\quad 0\leq\beta_{\epsilon}'(r)\leq1.
	\end{align}
	Then $|\beta_{\epsilon}(r)-r_+|\leq\epsilon$ for every $r\in\bR$, and $\beta_\epsilon'(r)\rightarrow\boldsymbol{1}_{\{r>0\}}$ at every $r$, including $r=0$.
	The chain rule gives
	\begin{align}
		\partial_{\xi\xi}\beta_\epsilon(h)=\beta_\epsilon'(h)h_{\xi\xi}+\beta_\epsilon''(h)h_\xi^2\geq\beta_\epsilon'(h)h_{\xi\xi},
	\end{align}
	\begin{align}
		\Delta\beta_\epsilon(\Theta h)=\beta_\epsilon'(\Theta h)\Delta(\Theta h)+\beta_\epsilon''(\Theta h)|\nabla(\Theta h)|^2\geq\beta_\epsilon'(\Theta h)\Delta(\Theta h).
	\end{align}
	Multiply both inequalities by $\chi$, integrate by parts, and pass to the limit as $\epsilon\rightarrow0$. Since $\Theta>0$, we have $(\Theta h)_+=\Theta h_+$ and $\{\Theta h>0\}=\{h>0\}$. By the dominated convergence theorem,
	\begin{align}
		\int_a^b\oint_{\bS^2}h_+\chi_{\xi\xi}\,d\mu\,d\xi\geq\int_a^b\oint_{\bS^2}\boldsymbol{1}_{\{h>0\}}\chi h_{\xi\xi}\,d\mu\,d\xi,
	\end{align}
	\begin{align}
		\int_a^b\oint_{\bS^2}\Theta h_+\Delta\chi\,d\mu\,d\xi\geq\int_a^b\oint_{\bS^2}\boldsymbol{1}_{\{h>0\}}\chi\Delta(\Theta h)\,d\mu\,d\xi.
	\end{align}
	Adding these inequalities and using \eqref{eq:comparison_principle_difference_equation}, we obtain \eqref{eq:comparison_principle_distribution_explaination}.
	Let $P(\xi)=\oint_{\bS^2}h_+(\xi,x)\,d\mu$, and take $\chi=\eta(\xi)$ with $0\leq\eta\in C_c^\infty((a,b))$. By \eqref{eq:comparison_principle_distribution_explaination}, we get
	$P''(\xi)\geq0$ in the distributional sense on $(a,b)$. By \eqref{eq:comparison_boundary_order}, $P$ extends continuously to $[a,b]$ with $P(a)=P(b)=0$. Since $P$ is convex and nonnegative, $P\equiv0$. Hence $h_+\equiv0$, so $w_1\leq w_2$.
\end{proof}

For the following corollaries, solutions of \eqref{eq:intro_toda} are assumed to be continuous up to $\xi=0$ and smooth in the interior.

The following corollary gives a uniform lower bound for $w$ in a fixed collar neighborhood of $\{\xi=0\}$, with both the width and the bound depending only on a lower bound for $\psi$.

\begin{corollary}\label{cor:boundary_lower}
	Let $w$ satisfy \eqref{eq:intro_toda}--\eqref{eq:intro_toda_boundary}, and let $0<a\leq\frac12$ satisfy $2\log a\leq\inf_{x\in\bS^2}\psi(x)$. Then
	\begin{align}\label{eq:toda_boundary_lower}
		w(\xi,x)\geq2\log(a-\xi),\quad 0\leq\xi<a,\quad x\in\bS^2.
	\end{align}
\end{corollary}

\begin{proof}
	The function $q(\xi,x)=2\log(a-\xi)$ solves \eqref{eq:intro_toda} on $(0,a)\times\bS^2$, and $q(0,x)\leq\psi(x)$. Since
	$$0\leq\oint_{\bS^2}(e^q-e^w)_+\,d\mu\leq(a-\xi)^2\rightarrow0\quad\text{as }\xi\rightarrow a,$$
	Lemma \ref{lem:comparison} gives $q\leq w$.
\end{proof}

The next corollary shows how a lower bound at a slice $\xi=\xi_2$ propagates towards $\xi=0$.

\begin{corollary}\label{cor:comparison_lower}
	Let $w$ satisfy \eqref{eq:intro_toda}--\eqref{eq:intro_toda_boundary} and \eqref{eq:intro_degenerate_end}. For any $0<\xi_1<\xi_2<\frac12$ and $x\in\bS^2$, we have
	\begin{align}\label{eq:toda_comparison_lower}
		w(\xi_1,x)\geq w(\xi_2,x)+2\log\frac{\xi_1}{\xi_2}.
	\end{align}
\end{corollary}

\begin{proof}
	For $\lambda>1$, define $w_\lambda(\xi,x)=w(\xi/\lambda,x)+2\log\lambda$. Then $w_\lambda$ solves \eqref{eq:intro_toda}, and $w_\lambda(0,x)\geq w(0,x)$. Moreover,
	$$0\leq\oint_{\bS^2}(e^w-e^{w_\lambda})_+\,d\mu\leq\oint_{\bS^2}e^w\,d\mu\rightarrow0\quad\text{as }\xi\rightarrow\frac12.$$
	Lemma \ref{lem:comparison} gives $w\leq w_\lambda$. Taking $\lambda=\xi_2/\xi_1$ and evaluating at $\xi=\xi_2$ proves the result.
\end{proof}

A quantitative lower estimate near the degenerate end $\xi=\frac12$ requires further arguments.

\begin{corollary}\label{cor:uniqueness}
	Let $w_1,w_2$ solve \eqref{eq:intro_toda} with $w_1(0,\cdot)=w_2(0,\cdot)$, each satisfying \eqref{eq:intro_degenerate_end}. Then $w_1=w_2$.
\end{corollary}

\begin{proof}
	The common boundary value gives
	$$\oint_{\bS^2}(e^{w_1}-e^{w_2})_+\,d\mu\rightarrow0\quad\text{as }\xi\rightarrow0.$$
	At the other endpoint,
	$$0\leq\oint_{\bS^2}(e^{w_1}-e^{w_2})_+\,d\mu\leq\oint_{\bS^2}e^{w_1}\,d\mu\rightarrow0\quad\text{as }\xi\rightarrow\frac12.$$
	Lemma \ref{lem:comparison}, applied with $a=0$ and $b=\frac12$, gives $w_1\leq w_2$. Interchanging $w_1$ and $w_2$ and repeating the argument gives $w_2\leq w_1$, hence $w_1=w_2$.
\end{proof}

\begin{corollary}\label{cor:evenness}
	Let $w$ satisfy \eqref{eq:intro_toda}--\eqref{eq:intro_toda_boundary} and \eqref{eq:intro_degenerate_end}. If $\psi$ is even, then $w$ is even.
\end{corollary}

\begin{proof}
	Since $\psi$ is even, $w(\xi,-x)$ satisfies the same equation, boundary data, and integral limit as $w(\xi,x)$. Corollary \ref{cor:uniqueness} therefore gives $w(\xi,-x)=w(\xi,x)$.
\end{proof}

To analyze Case 1, it is convenient to remove the degeneracy at $\xi=\frac12$ by making the following change of variables:
\begin{align}\label{eq:ball_transform}
	s=-\log(1-2\xi),\quad e^w=(\frac12-\xi)^2u(s,x),\quad \varphi(x)=4e^{\psi(x)}.
\end{align}
The $SU(\infty)$-Toda equation becomes
\begin{align}\label{eq:s_pde}
\Delta\log u+u_{ss}-3u_s+2u=2
\end{align}
on $[0,\infty)\times\bS^2$, with boundary value $u(0,x)=\varphi(x)$. The asymptotic condition in \eqref{eq:intro_ball_end} is equivalent to $u$ being bounded above and bounded away from zero. For the estimates below, it suffices to assume that $u>0$ is bounded, continuous up to $s=0$ and smooth in the interior.

Integrating \eqref{eq:s_pde} over $\bS^2$, we find that the function $a(s)=\oint_{\bS^2}u(s,x)\,d\mu$ satisfies $a''-3a'+2a=2$. Since $a$ is bounded, we conclude $a(s)\equiv 1$, so
\begin{align}\label{eq:mass_one}
\oint_{\bS^2}u(s,x)\,d\mu=1,\quad \oint_{\bS^2}\varphi\,d\mu=1.
\end{align}
Set
\begin{align}\label{eq:boundary_constants}
m_0=\inf_{x\in\bS^2}\varphi(x),\quad M_0=\sup_{x\in\bS^2}\varphi(x),\quad \Ent_0=\oint_{\bS^2}\varphi\log\varphi\,d\mu.
\end{align}
Then $0<m_0\leq 1\leq M_0$ and $0\leq \Ent_0\leq\log M_0$.

Corollary \ref{cor:comparison_lower}, with $\xi_1=(1-e^{-s})/2$ and $\xi_2=(1-e^{-S})/2$, gives, for $0<s<S$,
\begin{align}\label{eq:comparison_lower}
	u(s,x)\geq\left(\frac{e^s-1}{e^S-1}\right)^2u(S,x).
\end{align}
Similarly, taking $a=\sqrt{m_0}/2$ in Corollary \ref{cor:boundary_lower} gives, for $0\leq s\leq s_0=-\log(1-\sqrt{m_0}/2)$,
\begin{align}\label{eq:boundary_lower}
	u(s,x)\geq(1-(1-\sqrt{m_0})e^s)^2\geq\frac{m_0}{4}.
\end{align}

By \eqref{eq:ball_transform} and Corollary \ref{cor:uniqueness}, two positive solutions $u_1,u_2$ of \eqref{eq:s_pde}, continuous up to $s=0$ and smooth in the interior, coincide if they have the same boundary value and satisfy
\begin{align}\label{eq:uniqueness_growth}
	\lim\limits_{s\rightarrow\infty}e^{-2s}\oint_{\bS^2}u_i(s,x)\,d\mu=0,\quad i=1,2.
\end{align}
Indeed, \eqref{eq:ball_transform} identifies this condition with \eqref{eq:intro_degenerate_end}. In particular, uniqueness holds for bounded positive solutions. If $\varphi$ is even, then any solution satisfying \eqref{eq:uniqueness_growth} is even by Corollary \ref{cor:evenness}.

\section{A priori estimates for even solutions}\label{sec:even_estimates}

In this section, let $u>0$ be a solution of \eqref{eq:s_pde}, continuous up to $s=0$ and smooth in the interior. For convenience in the a priori estimates, we impose the stronger asymptotic condition
\begin{align}\label{eq:boundary_asymptotics}
u(0,x)=\varphi(x),\quad \lim\limits_{s\rightarrow\infty}\sup_{x\in\bS^2}|u(s,x)-1|=0.
\end{align}
For non-even boundary data $\varphi$, the limiting function need not be $1$; we treat this case in Section \ref{sec:general_case}.

We also assume that $u$ is even, that is, $u(s,-x)=u(s,x)$ for every $(s,x)$. Set
\begin{align}\label{eq:entropy_moments}
\Ent(s)=\oint_{\bS^2}u(s,x)\log u(s,x)d\mu,\quad F_q(s)=\oint_{\bS^2}u^q(s,x)d\mu.
\end{align}

By \eqref{eq:mass_one}, $\Ent(s)=\oint_{\bS^2}(u(s,x)\log u(s,x)-u(s,x)+1)d\mu \geq 0$, $\Ent(0)=\Ent_0$, and $F_q\geq 1$ for $q\geq 1$. For $1\leq p<q$, H\"older's inequality gives $F_p\leq F_q^{p/q}\leq F_q$. The asymptotic condition \eqref{eq:boundary_asymptotics} implies that $u$ is bounded, $\Ent(s)\rightarrow 0$, and $F_q(s)\rightarrow 1$ for each finite $q\geq 1$.

\subsection{$L^2$ bound and upper bound}

We begin with the following inequality for even functions on $\bS^2$.

\begin{lemma}[Gagliardo--Nirenberg inequality]\label{lem:improved_gns}
 There are universal constants $0<A_e<2$ and $B_e\geq 1$ such that every even function $\psi\in H^1(\bS^2)$ satisfies
 \begin{align}\label{eq:improved_gns}
	\oint_{\bS^2}\psi^4d\mu\leq A_e\left(\oint_{\bS^2}|\nabla\psi|^2d\mu\right)\left(\oint_{\bS^2}\psi^2d\mu\right)+B_e\left(\oint_{\bS^2}\psi^2d\mu\right)^2.
 \end{align}
\end{lemma}

\begin{proof}
	We use the sharp Euclidean Gagliardo--Nirenberg inequality in \cite[p.~568]{Weinstein1983}
	$$
	\|f\|_{L^4(\bR^2)}^4\leq C_{\bR^2}\|\nabla f\|_{L^2(\bR^2)}^2\|f\|_{L^2(\bR^2)}^2, \,\, f\in C_c^\infty(\mathbb{R}^2),
	$$
	where $C_{\bR^2}=\frac{1}{\pi\cdot 1.86225...}$ numerically.
	
	Since $\psi$ is even, it descends to a function on $\mathbb{RP}^2$. Let $dA$ denote the unnormalized area measure of the induced round metric on $\mathbb{RP}^2$. By the standard localization argument for Gagliardo--Nirenberg inequalities on compact manifolds, for every $\epsilon>0$ there is $C_\epsilon$ such that
	$$
	\int_{\mathbb{RP}^2}\psi^4\,dA\leq\left(\frac{1}{\pi\cdot 1.86225...}+\epsilon\right)
	\left(\int_{\mathbb{RP}^2}|\nabla\psi|^2\,dA\right)
	\left(\int_{\mathbb{RP}^2}\psi^2\,dA\right)
	+C_\epsilon\left(\int_{\mathbb{RP}^2}\psi^2\,dA\right)^2.
	$$
	The round metric on $\mathbb{RP}^2$ has area $2\pi$. Passing to normalized area measure therefore gives
	$$
	\oint_{\bS^2}\psi^4d\mu\leq
	\left(\frac{2}{1.86225...}+2\pi\epsilon\right)
	\left(\oint_{\bS^2}|\nabla\psi|^2d\mu\right)
	\left(\oint_{\bS^2}\psi^2d\mu\right)
	+B_\epsilon\left(\oint_{\bS^2}\psi^2d\mu\right)^2.
	$$
	Since $\frac{2}{1.86225...}<2$, choosing $\epsilon>0$ sufficiently small proves the result.
\end{proof}

\begin{lemma}\label{lem:second_moment_ode}
	Set $\delta_e=8/A_e-4>0$ and $C_e=8B_e/A_e-4>0$. Then
	\begin{align}\label{eq:second_moment_ode}
		F_2''-3F_2'-\delta_e F_2\geq-C_e.
	\end{align}
	In particular, with $K_e=C_e/\delta_e\geq1$, we have
	\begin{align}\label{eq:second_moment_bound}
		F_2(s)\leq\max\{F_2(0),K_e\}.
	\end{align}
\end{lemma}

\begin{proof}
	Using \eqref{eq:mass_one} and \eqref{eq:improved_gns} applied to $\sqrt{u}$, we have
	\begin{align}
		F_2''-3F_2'+4F_2&=4+2\oint_{\bS^2}(u_s^2+\frac{|\nabla u|^2}{u})d\mu\\
		&\geq4+8\oint_{\bS^2}|\nabla\sqrt{u}|^2d\mu\\
		&\geq4+\frac{8}{A_e}(F_2-B_e)=\frac{8}{A_e}F_2+4-\frac{8B_e}{A_e}.\label{eq:second_moment_gns_step}
	\end{align}
	Rearranging gives \eqref{eq:second_moment_ode} by the definitions of $\delta_e$ and $C_e$.
	A maximum of $F_2$ above $\max\{F_2(0),K_e\}$ would contradict \eqref{eq:second_moment_ode}, since $F_2(s)\rightarrow1\leq K_e$.
\end{proof}

We now use \eqref{eq:second_moment_bound} to obtain a uniform upper bound for $u$ by a Moser--type iteration.

\begin{proposition}\label{prop:upper_bound}
	Let $u$ be a positive even classical solution of \eqref{eq:s_pde} and \eqref{eq:boundary_asymptotics}. With $M_0$ as in \eqref{eq:boundary_constants}, there is a universal constant $C_*\geq1$ such that
	\begin{align}\label{eq:upper_bound}
		\|u\|_{L^\infty([0,\infty)\times\bS^2)}\leq C_*M_0^4.
	\end{align}
\end{proposition}
	
\begin{proof}
	A direct computation shows that for $q>1$,
	\begin{align}\label{eq:moment_identity}
		F_q''-3F_q'+2qF_q
		&=2qF_{q-1}+q(q-1)\oint_{\bS^2}u^{q-2}u_s^2d\mu+\frac{4q}{q-1}\oint_{\bS^2}|\nabla u^{\frac{q-1}{2}}|^2d\mu.
	\end{align}
	The Sobolev inequality on $\bS^2$ gives
	$$F_q''-3F_q'+2qF_q\geq\frac{4q}{q-1}(C^{-1}F_{3(q-1)}^{1/3}-F_{q-1}),$$
	where $C$ is universal. In this proof, $C$ may increase from line to line.
	
	Let $N_q=\sup_{s\geq0}F_q(s)^{1/q}$. For $q>3$, we show that
     \begin{align}	\label{eq: M_q iteration}
	N_q	\leq \max\{\|\varphi\|_{L^\infty(\bS^2)},	(Cq^3)^{1/(2q-6)}	N_{q/2}^{(2q-3)/(2q-6)}\}.
     \end{align}
	
	Since $\lim\limits_{s\rightarrow\infty}F_q(s)=1$ and $F_q(s)\geq1$, $F_q$ achieves its maximum at some $s_*\in[0,\infty)$. If $s_*=0$, then $F_q(s)\leq F_q(0)\leq\|\varphi\|_{L^\infty(\bS^2)}^q$. Otherwise, $s_*>0$, and
	$$F_q'(s_*)=0,\quad  F_q''(s_*)\leq 0.$$
	It follows that
	
	$$F_{3(q-1)}^{1/3}(s_*)\leq C((q-1)F_q(s_*)+F_{q-1}(s_*))\leq  Cq F_q(s_*),$$
	where we used $F_{q-1}\leq F_{q}$. Moreover, by H\"older's inequality,
	$$F_q^{1/q}(s_*)	\leq F_{3(q-1)}^{1/(5q-6)}(s_*)	F_{q/2}^{2(2q-3)/(q(5q-6))}(s_*),$$
	so for $q>3$,
	$$	F_q(s_*)	\leq	(Cq^3)^{q/(2q-6)}	F_{q/2}^{(2q-3)/(q-3)}(s_*).$$
	Considering the two possible locations of $s_*$ gives \eqref{eq: M_q iteration}.
	Iterate \eqref{eq: M_q iteration} for $q=4,8,16,\cdots$, and note
	$\prod_{j=2}^\infty\frac{2^{j+1}-3}{2(2^j-3)}=4$.
	The accumulated constant factor is uniformly bounded. By \eqref{eq:second_moment_bound} and $M_0\geq1$,
	$$N_2\leq \max\{F_2(0)^{\frac12},K_e^{\frac12}\}\leq \max\{\|\varphi\|_{L^\infty(\bS^2)},K_e^{\frac12}\}\leq C\|\varphi\|_{L^\infty(\bS^2)}. $$
	It follows that $\limsup\limits_{j\rightarrow\infty}N_{2^j}\leq CM_0^4$. Thus
	$$\|u(s,\cdot)\|_{L^\infty(\bS^2)}=\lim\limits_{j\rightarrow\infty}F_{2^j}(s)^{1/2^j}\leq\limsup\limits_{j\rightarrow\infty}N_{2^j}\leq CM_0^4.$$
\end{proof}

\subsection{Entropy decay and lower bound}

The upper bound in Proposition \ref{prop:upper_bound} allows us to take $\Lambda=C_*M_0^4$ in the estimates below. We first use the Onofri inequality to obtain quantitative entropy decay.

\begin{lemma}[Sharp Onofri inequality]\label{lem:improved_onofri} Let $\psi\in H^1(\bS^2)$ be an even function. Then
	\begin{align}\label{eq:improved_onofri}
		\log\left(\oint_{\bS^2}e^\psi d\mu\right)\leq \oint_{\bS^2}\psi\,d\mu+\frac18\oint_{\bS^2}|\nabla\psi|^2d\mu.
	\end{align}
\end{lemma}

This is the even-function version of the Onofri inequality; see \cite[Corollary 2.2, equation (2.40), p. 168]{OsgoodPhillipsSarnak1988}. We will use the sharp constant $1/8$ for the argument of the even case.

\begin{lemma} \label{lem:entropy-onofri inequality}
	Let $\psi>0$ be an even function such that $\log\psi\in H^1(\bS^2)$ and $\oint_{\bS^2}\psi\,d\mu=1$. Then
	\begin{align}\label{eq:entropy_onofri}
		\oint_{\bS^2}|\nabla\log\psi|^2d\mu\geq2\oint_{\bS^2}\psi\log\psi\,d\mu-4\oint_{\bS^2}\log\psi\,d\mu.
	\end{align}
\end{lemma}

\begin{proof}
	Set $g=\log\psi-\oint_{\bS^2}\log\psi\,d\mu$. The function $\Phi(t)=\log(\oint_{\bS^2}e^{tg}d\mu)$ is convex, so for $t\geq1$, $(t-1)\Phi'(1)\leq \Phi(t)-\Phi(1)$.
	By \eqref{eq:improved_onofri},
	$$\Phi(t)\leq\frac{t^2}{8}\oint_{\bS^2}|\nabla\log\psi|^2d\mu.$$
	Since $\oint_{\bS^2}\psi\,d\mu=1$, we have $\Phi'(1)=\oint_{\bS^2}(\psi-1)\log\psi\,d\mu$ and $\Phi(1)=-\oint_{\bS^2}\log\psi\,d\mu$. It follows that, for every $t\geq1$,
	$$\oint_{\bS^2}|\nabla\log\psi|^2d\mu\geq\frac{8(t-1)}{t^2}\oint_{\bS^2}(\psi-1)\log\psi\,d\mu-\frac{8}{t^2}\oint_{\bS^2}\log\psi\,d\mu.$$
	Taking $t=2$ gives \eqref{eq:entropy_onofri}.
\end{proof}

\begin{lemma}\label{lem:entropy_ode}
	If $0<u\leq\Lambda$ on $[0,\infty)\times\bS^2$, then
	\begin{align}\label{eq:entropy_ode}
		\Ent''-3\Ent'-\frac{2}{\Lambda}\Ent\geq 0.
	\end{align}
\end{lemma}

\begin{proof}
	\begin{align}
		\Ent''&=\oint_{\bS^2}\left(\frac{u_s^2}{u}+|\nabla\log u|^2\right)d\mu+3\oint_{\bS^2}(1+\log u)u_s\,d\mu+2\oint_{\bS^2}(1-u)(1+\log u)d\mu\\
		&=\oint_{\bS^2}\left(\frac{u_s^2}{u}+|\nabla\log u|^2\right)d\mu+3\Ent'-2\Ent+2\oint_{\bS^2}\log u\,d\mu.
	\end{align}
	It suffices to show 
	\begin{align}
		\oint_{\bS^2}|\nabla\log u|^2d\mu\geq2(1+\frac{1}{\Lambda})\Ent-2\oint_{\bS^2}\log u\,d\mu.
	\end{align}
	By Lemma \ref{lem:entropy-onofri inequality}, we have
	\begin{align}\label{eq:entropy_onofri_step}
		\oint_{\bS^2}|\nabla\log u|^2d\mu\geq2\Ent-4\oint_{\bS^2}\log u\,d\mu.
	\end{align}
	So it suffices to show
	\begin{align}
		-\oint_{\bS^2}\log u\,d\mu\geq\frac{1}{\Lambda}\Ent,
	\end{align}
	which follows by integrating the pointwise inequality
	$$r\log r-r+1\leq\Lambda(r-1-\log r),\quad 0<r\leq\Lambda,$$
	and using \eqref{eq:mass_one}. The pointwise inequality follows by comparing second derivatives and noting that both sides and their first derivatives vanish at $r=1$.
\end{proof}

\begin{corollary}\label{cor:entropy_decay} Set
\begin{align}\label{eq:entropy_rate}
\beta_\Lambda=\frac{\sqrt{9+8/\Lambda}-3}{2}.
\end{align}
If $0<u\leq\Lambda$ on $[0,\infty)\times\bS^2$, then
	\begin{align}\label{eq:entropy_decay}
		0\leq\Ent(s)\leq \Ent_0e^{-\beta_\Lambda s}.
	\end{align}
	Moreover, for every $q>1$,
	\begin{align}\label{eq:moment_decay}
		0\leq F_q(s)-1\leq q(q-1)\Lambda^{q-1}\Ent_0e^{-\beta_\Lambda s}.
	\end{align}
\end{corollary}

\begin{proof}
	The function $\Ent_0e^{-\beta_\Lambda s}$ solves the homogeneous equation associated with \eqref{eq:entropy_ode}. Since $\Ent(s)\rightarrow0$, a positive maximum of $\Ent(s)-\Ent_0e^{-\beta_\Lambda s}$ would occur at an interior point and contradict \eqref{eq:entropy_ode}. This proves \eqref{eq:entropy_decay}.

	To prove \eqref{eq:moment_decay}, we integrate the following pointwise inequality and use \eqref{eq:entropy_decay}: 
	$$r^q-1-q(r-1)\leq q(q-1)\Lambda^{q-1}(r\log r-r+1),\quad 0<r\leq\Lambda.$$
	The pointwise inequality follows by comparing second derivatives and noting that both sides and their first derivatives vanish at $r=1$. 
	
\end{proof}

To obtain a lower bound from the entropy decay, we establish a local De Giorgi-type estimate. This local estimate does not require evenness or a previously known upper or lower bound.

Write $Q_r(S)=(S-r,S+r)\times\bS^2$ for $0<r\leq S$.

\begin{proposition}\label{prop:local_de_giorgi}
	There exists a universal constant $\epsilon_*>0$ such that, if $S\geq2$ and $u>0$ is a smooth solution of \eqref{eq:s_pde} on $Q_2(S)$ satisfying
	$$|\{u<\frac12\}\cap Q_2(S)|<\epsilon_*,$$
	then
	$$u\geq\frac14\quad\mathrm{in}\quad Q_1(S).$$
\end{proposition}

\begin{proof}
	Write the equation in divergence form
	\begin{align}\label{eq:weighted_divergence}
		\mathrm{div}(u^{-1}\nabla u)+u_{ss}-3u_s+2(u-1)=0.
	\end{align}
	Set $$r_j=1+2^{-j},\quad  k_j=\frac14+2^{-j-2}, \quad w_j=(k_j-u)_+,\quad A_j=Q_{r_j}(S)\cap\{u<k_j\},$$
	so $$A_0=\{u<\frac12\}\cap Q_2(S),\quad  \cap_{j=0}^\infty A_j=\overline{Q_1(S)}\cap \{u\leq \frac14\}.$$
	Choose cutoff functions $0\leq\eta_j=\eta_j(s)\leq1$ in $C_c^\infty(\bR)$ such that
	$$\mathrm{supp}(\eta_j)\subset(S-r_j,S+r_j),\quad\eta_j|_{[S-r_{j+1},S+r_{j+1}]}\equiv1,\quad|\eta_j'|\leq C2^j.$$
	Multiplying \eqref{eq:weighted_divergence} by $\eta_j^2w_j$ and integrating by parts gives
	\begin{align}
		&\int\oint_{Q_2(S)}\eta_j^2(u^{-1}\nabla u\cdot\nabla w_j+u_s\,(w_j)_s)d\mu\,ds\\
		&=-2\int\oint_{Q_2(S)}\eta_j\eta_j'w_ju_s\,d\mu\,ds-3\int\oint_{Q_2(S)}\eta_j^2w_ju_s\,d\mu\,ds\\
		&\quad+2\int\oint_{Q_2(S)}\eta_j^2(u-1)w_j\,d\mu\,ds.
	\end{align}
	Since $$u^{-1}\geq 2 \,\, \mathrm{on}\,\,  A_j, \quad |w_j|\leq\frac12,\quad |(u-1)w_j|\leq \frac12 \,\, \mathrm{on}\,\, Q_2(S), $$
	$$\nabla w_j=-\boldsymbol{1}_{\{u<k_j\}}\nabla u,\quad   (w_j)_s=-\boldsymbol{1}_{\{u<k_j\}}u_s \,\,\,\, \mathrm{a.e.},$$
	we have
	\begin{align}
		\int\oint_{A_j}\eta_j^2(|\nabla u|^2+u_s^2)d\mu\,ds
		\leq C2^j\int\oint_{A_j}|\eta_ju_s|d\mu\,ds+C|A_j|.
	\end{align}
	By Young's inequality, the first term on the right is at most
	$$\frac12\int\oint_{A_j}\eta_j^2u_s^2d\mu\,ds+C4^j|A_j|.$$
	Absorbing the gradient term, we obtain
	\begin{align}\label{eq:sublevel_energy}
		\int\oint_{A_j}\eta_j^2(|\nabla u|^2+u_s^2)d\mu\,ds\leq C4^j|A_j|.
	\end{align}
	Since $\mathrm{supp}(\eta_jw_j)\subset Q_2(S)$, the three-dimensional Sobolev inequality gives
	\begin{align}
		\left(\int\oint_{Q_{r_j}(S)}|\eta_jw_j|^6d\mu\,ds\right)^{1/3}
		&\leq C\int\oint_{Q_{r_j}(S)}(|\nabla(\eta_jw_j)|^2+|\partial_s(\eta_jw_j)|^2)d\mu\,ds\\
		&\leq C\int\oint_{A_j}\eta_j^2(|\nabla u|^2+u_s^2)d\mu\,ds+C4^j|A_j|\\
		&\leq C4^j|A_j|,
	\end{align}
	where we used \eqref{eq:sublevel_energy}. If $(s,x)\in A_{j+1}$, then $\eta_j=1$ and $w_j\geq k_j-k_{j+1}=2^{-j-3}$. It follows that
	\begin{align}\label{eq:sublevel_iteration}
	(2^{-6(j+3)}|A_{j+1}|)^{1/3}\leq C4^j|A_j|,\quad |A_{j+1}|\leq C2^{12j}|A_j|^3.
	\end{align}
	Hence, if $|A_0|<\epsilon_*:=(8\sqrt{C})^{-1}$, we have $\lim\limits_{j\rightarrow \infty}|A_j|=0$, so $|\overline{Q_1(S)}\cap \{u\leq \frac14\}|=0.$
	Since $u$ is continuous on $Q_1(S)$, we conclude $u\geq \frac{1}{4}$ on $Q_1(S)$.
\end{proof}

\begin{proposition}\label{prop:tail_lower}
	For fixed boundary data $\varphi$, assume $0<u\leq\Lambda$. If $\Ent_0>0$, define
	\begin{align}\label{eq:tail_threshold}
	S_\Lambda=\max\left\{2,\,2+\beta_\Lambda^{-1}\log\left(\frac{32\Ent_0}{\epsilon_*}\right)\right\},
	\end{align}
	where $\beta_\Lambda$ is given by \eqref{eq:entropy_rate} and $\epsilon_*$ by Proposition \ref{prop:local_de_giorgi}. If $\Ent_0=0$, set $S_\Lambda=2$. Then
	\begin{align}\label{eq:tail_lower}
	u\geq\frac14\quad\mathrm{on}\quad[S_\Lambda,\infty)\times\bS^2.
	\end{align}
\end{proposition}

\begin{proof}
If $\Ent_0=0$, then \eqref{eq:mass_one} and \eqref{eq:entropy_decay} give $u\equiv1$, so the conclusion is immediate. Assume therefore that $\Ent_0>0$.
For $0<r\leq1/2$, we have $r\log r-r+1\geq(r-1)^2/2\geq1/8$. The first inequality follows from Taylor's theorem at $r=1$. Hence, by \eqref{eq:mass_one} and \eqref{eq:entropy_decay}, for $S\geq2$ we have
\begin{align}\label{eq:tail_sublevel_measure}
	|\{u<\frac{1}{2}\}\cap Q_2(S)|
	&\leq8\int_{S-2}^{S+2}\Ent(s)\,ds\\
	&\leq8\Ent_0\int_{S-2}^{S+2}e^{-\beta_\Lambda s}\,ds\\
	&<32\Ent_0e^{-\beta_\Lambda(S-2)}.
\end{align}
For $S\geq S_\Lambda\geq2$, the last expression is at most $\epsilon_*$. Proposition \ref{prop:local_de_giorgi} therefore gives $u\geq1/4$ on $Q_1(S)$, proving \eqref{eq:tail_lower}.
\end{proof}

Combining the comparison estimates in Corollaries \ref{cor:boundary_lower} and \ref{cor:comparison_lower} with the tail lower bound in Proposition \ref{prop:tail_lower} gives a global lower bound. Together with the upper bound in Proposition \ref{prop:upper_bound}, this yields the following a priori $C^0$ estimate.

\begin{theorem}\label{thm:c0_estimate}
	Let $u$ be a positive even classical solution of \eqref{eq:s_pde} and \eqref{eq:boundary_asymptotics}. With $m_0,M_0,\Ent_0$ as in \eqref{eq:boundary_constants}, set $\Lambda=C_*M_0^4$, where $C_*$ is the universal constant in Proposition \ref{prop:upper_bound}, and define $S_\Lambda$ by \eqref{eq:tail_threshold}. Then
	\begin{align}\label{eq:global_c0_estimate}
	\frac{m_0}{16}e^{-2S_\Lambda}\leq u\leq\Lambda\quad\mathrm{on}\quad[0,\infty)\times\bS^2.
	\end{align}
	In particular, both bounds are controlled by $m_0$ and $M_0$, since $0\leq \Ent_0\leq\log M_0$.
\end{theorem}

\begin{proof}
	The upper bound is \eqref{eq:upper_bound}. By \eqref{eq:boundary_lower}, $u\geq m_0/4$ on $[0,s_0]\times\bS^2$, where $s_0=-\log(1-\sqrt{m_0}/2)$. Since $s_0\leq\log2<S_\Lambda$, \eqref{eq:comparison_lower} and \eqref{eq:tail_lower} give, for $s_0\leq s\leq S_\Lambda$,
	\begin{align}
		u(s,x)\geq\frac14\left(\frac{e^{s_0}-1}{e^{S_\Lambda}-1}\right)^2\geq\frac{m_0}{16}e^{-2S_\Lambda}.
	\end{align}
	Together with $u\geq1/4$ on $[S_\Lambda,\infty)\times\bS^2$, this proves the lower bound. 
\end{proof}

\subsection{Derivative estimates and decay}

We use the following Gagliardo--Nirenberg interpolation inequality on the three-dimensional cylinder.

\begin{lemma}\label{lem:interpolation}
	For integers $0\leq j<k$, set $\theta=\frac{2(k-j)}{2k+3}$. For every $f\in C^k([0,1]\times\bS^2)$, we have
	$$\|f\|_{C^j([0,1]\times \bS^2)}\leq C_{j,k} \|f\|_{C^k([0,1]\times\bS^2)}^{1-\theta} \|f\|_{L^2([0,1]\times\bS^2)}^{\theta}.$$
\end{lemma}

\begin{proof}
	This follows from the standard Gagliardo--Nirenberg inequality in dimension three, applied in local charts with extension across the boundary; the lower-order derivatives are controlled using $\|f\|_{L^2}\leq \|f\|_{C^k}$.
\end{proof}

\begin{corollary}\label{cor:derivative_decay}
	Under the assumptions of Theorem \ref{thm:c0_estimate}, for each integer $k\geq0$ and $0<\alpha<1$, there is a constant $C>0$, depending only on $m_0,M_0,\Ent_0,k,\alpha$, such that
	\begin{align}\label{eq:derivative_decay}
	\|u-1\|_{C^{k,\alpha}([r,r+1]\times\bS^2)}\leq Ce^{-\delta_k r},\quad r\geq1,\quad\delta_k=\frac{\beta_\Lambda}{2k+7}.
	\end{align}
	If $\varphi\in C^{2,\alpha}(\bS^2)$, there is also a uniform bound for $\|u\|_{C^{2,\alpha}([0,2]\times\bS^2)}$, with additional dependence on $\|\varphi\|_{C^{2,\alpha}(\bS^2)}$.
\end{corollary}

\begin{proof}
	By \eqref{eq:global_c0_estimate}, $u^{-1}$ is bounded above and below by positive constants. Thus the divergence form equation
	\begin{align}\label{eq:divergence_pde}
	\mathrm{div}(u^{-1}\nabla u)+u_{ss}-3u_s+2u=2
	\end{align}
	is uniformly elliptic when $u^{-1}$ is regarded as a given coefficient function. The interior De Giorgi--Nash--Moser estimate \cite[Theorem 8.22]{GilbargTrudinger2001}, followed by Schauder estimates, gives
	$$\|u\|_{C^{k+2}([r,r+1]\times\bS^2)}\leq C_k,\quad r\geq1.$$
	On the other hand, \eqref{eq:mass_one} and \eqref{eq:moment_decay} give
	\begin{align}\label{eq:l2_decay}
	\|u-1\|_{L^2([r,r+1]\times\bS^2)}^2=\int_r^{r+1}(F_2(s)-1)ds\leq2\Lambda \Ent_0e^{-\beta_\Lambda r}.
	\end{align}
	Applying Lemma \ref{lem:interpolation} with orders $k+1$ and $k+2$, for which $\theta=2/(2k+7)$, yields
	$$\|u-1\|_{C^{k,\alpha}([r,r+1]\times\bS^2)}\leq C\|u-1\|_{C^{k+1}([r,r+1]\times\bS^2)}\leq Ce^{-\delta_k r}.$$
	The estimate up to $s=0$ follows from the boundary De Giorgi--Nash--Moser estimate \cite[Theorem 8.29]{GilbargTrudinger2001} and Schauder estimates with boundary data $\varphi\in C^{2,\alpha}(\bS^2)$.
\end{proof}

\section{Existence in the even case}\label{sec:even_existence}

The aim of this section is to prove the following.

\begin{theorem}[Even case]\label{thm:existence_uniqueness}
For any positive even function $\varphi\in C^{2,\alpha}(\bS^2)$, $0<\alpha<1$, satisfying $\oint_{\bS^2}\varphi\,d\mu=1$, there is a unique positive solution $u$ of \eqref{eq:s_pde} and \eqref{eq:boundary_asymptotics}. Moreover, $u$ is even, and there exist constants $C>1,\delta\in (0,1)$, depending only on $\varphi$, $\alpha$, such that
\begin{align}\label{eq:even_existence_bounds}
	C^{-1}\leq u\leq C,\quad \sup_{s\geq 0}e^{\delta s}\|u-1\|_{C^{2,\alpha}([s,s+1]\times\bS^2)}\leq C.
\end{align}
\end{theorem}

The weighted H\"older spaces used below are standard. For completeness, we give a self-contained treatment of the weighted-space arguments needed for the continuity method.

\subsection{Weighted H\"older spaces}

For integers $j\geq 0$, $\alpha\in(0,1)$, and $\delta>0$, define the weighted H\"older space $C_\delta^{j,\alpha}$ to consist of functions for which the norm
\begin{align}\label{eq:weighted_norm}
	\|h\|_{C_\delta^{j,\alpha}}=\sup_{s\geq0}e^{\delta s}\|h\|_{C^{j,\alpha}([s,s+1]\times\bS^2)}
\end{align}
is finite.

Let
\begin{align}\label{eq:x_delta}
	X_\delta=\{h\in C_\delta^{2,\alpha}|\,h\text{ is even},\,h(0,\cdot)=0,\,\oint_{\bS^2}h(s,x)d\mu=0,\,\forall s\geq0\},
\end{align}
\begin{align}\label{eq:y_delta}
	Y_\delta=\{f\in C_\delta^{0,\alpha}|\,f\text{ is even},\,\oint_{\bS^2}f(s,x)d\mu=0,\,\forall s\geq0\},
\end{align}
equipped with the induced norms. These are closed subspaces of the corresponding weighted H\"older spaces and hence are Banach spaces.

\begin{lemma}\label{lem:weighted_local_compactness}
Let $h_i\in X_\delta$ satisfy $\|h_i\|_{X_\delta}\leq B$. There is a subsequence and a function $h_\infty\in X_\delta$ such that $h_i\rightarrow h_\infty$ in $C^{2,\beta}([0,R]\times\bS^2)$ for every $R>0$ and every $0<\beta<\alpha$. Moreover, $\|h_\infty\|_{X_\delta}\leq B$.
\end{lemma}

\begin{proof}
	The weighted bound gives uniform $C^{2,\alpha}$ bounds on compact cylinders. Arzelà--Ascoli and a diagonal argument therefore give local $C^{2,\beta}$ convergence, for every $0<\beta<\alpha$. By lower semicontinuity,
	$e^{\delta s}\|h_\infty\|_{C^{2,\alpha}([s,s+1]\times\bS^2)}\leq B	$
	for every $s\geq0$, hence $\|h_\infty\|_{X_\delta}\leq B$. The remaining defining conditions of $X_\delta$ pass to the limit.
\end{proof}

\subsection{The linearized operator}

For a positive function $u$, set
\begin{align}\label{eq:nonlinear_operator}
\mathcal{P}(u)=\Delta\log u+u_{ss}-3u_s+2u-2.
\end{align}
For a positive even function $u\in1+C_\delta^{2,\alpha}$, the linearization of $\mathcal{P}$ at $u$ is
\begin{align}\label{eq:linearized_operator}
L_uv=\Delta(u^{-1}v)+v_{ss}-3v_s+2v.
\end{align}

\begin{lemma}\label{lem:linear_isomorphism}
	Let $0<\delta<1$ and $0<\alpha<1$. Suppose that $u>0$ is even and $u-1\in C_\delta^{2,\alpha}$. Then
	$$L_u:X_\delta\rightarrow Y_\delta,\quad L_uv=\Delta(u^{-1}v)+v_{ss}-3v_s+2v,$$
	is an isomorphism of Banach spaces.
\end{lemma}

\begin{proof}
	We first prove injectivity, then an a priori estimate and surjectivity for $L_1$, and finally surjectivity for $L_u$ by the method of continuity.
	Since $u>0$ and $u\rightarrow1$ uniformly as $s\rightarrow\infty$, there is $c>0$ such that $c\leq u\leq c^{-1}$. Thus $L_u$ is uniformly elliptic. Moreover, if $a(s)=\oint_{\bS^2}v(s,x)d\mu$, then
	$$\oint_{\bS^2}L_uv\,d\mu=a''-3a'+2a,$$
	so $L_u$ maps $X_\delta$ continuously into $Y_\delta$.
	
	We first prove injectivity. Suppose $L_uv=0$. With $s=-\log(1-2\xi)$ as in \eqref{eq:ball_transform}, set
	$$h(\xi,x)=(\frac12-\xi)^2v(s,x),\quad \Theta(\xi,x)=\frac{1}{(\frac12-\xi)^2u(s,x)}>0.$$
	Then
	$$L_\Theta h:=h_{\xi\xi}+\Delta(\Theta h)=0.$$
	Since $v(0,\cdot)=0$ and $v$ is bounded, $h$ extends continuously to $[0,\frac12]\times\bS^2$ with $h(0,\cdot)=h(\frac12,\cdot)=0$. Applying the comparison argument in the proof of Lemma \ref{lem:comparison} to $h$ and $-h$ gives $h\equiv0$, hence $v\equiv0$.

We next prove an a priori estimate for the model operator
$$L_1=\Delta+\partial_{ss}-3\partial_s+2.$$
Let $Y_{j,m}$, $j=0,1,\ldots$ and $0\leq m\leq 2j$, be a real orthonormal basis of spherical harmonics in $L^2(\bS^2,d\mu)$, satisfying
$$-\Delta Y_{j,m}=\lambda_jY_{j,m},\quad \lambda_j=j(j+1).$$
Fix any $v\in X_\delta$ and set $f=L_1v\in Y_\delta$. Let $v_{j,m}$ and $f_{j,m}$ denote their coefficients in this basis. Since functions in $X_\delta$ and $Y_\delta$ are even and have zero spherical mean, only degrees $j=2,4,\ldots$ occur, i.e., $v_{j,m}=f_{j,m}=0$ if $j$ is odd. The equation $L_1v=f$ becomes
$$v_{j,m}''-3v_{j,m}'+(2-\lambda_j)v_{j,m}=f_{j,m},\quad v_{j,m}(0)=0,\quad v_{j,m}(s)\rightarrow0.$$
The characteristic roots are $j+2$ and $1-j$. The corresponding Dirichlet Green kernel is
$$G_j(s,t)=-\frac{e^{\frac32(s-t)}}{2j+1}\left(e^{-(j+\frac12)|s-t|}-e^{-(j+\frac12)(s+t)}\right).$$
The boundary conditions determine each mode uniquely:
\begin{align}\label{eq:model_green_formula}
	v_{j,m}(s)=\int_0^\infty G_j(s,t)f_{j,m}(t)\,dt.
\end{align}
The Green kernel satisfies
\begin{align}\label{eq:model_green_bound}
	|G_j(s,t)|
	\leq \frac{e^{-(j+\frac12)|t-s|+\frac32(s-t)}}{2j+1}
	\leq e^{-|t-s|},
\end{align}
where the last inequality follows from $j\geq2$. Parseval's identity and Minkowski's inequality give
\begin{align}
	\|v(s,\cdot)\|_{L^2(\bS^2)}
	&\leq \int_0^\infty e^{-|t-s|}
	\|f(t,\cdot)\|_{L^2(\bS^2)}\,dt.
\end{align}
Multiplying by $e^{\delta s}$ and using $0<\delta<1$, we obtain
\begin{align}\label{eq:model_weighted_l2}
	\sup_{s\geq0}e^{\delta s}\|v(s,\cdot)\|_{L^2(\bS^2)}
	\leq C_\delta
	\sup_{s\geq0}e^{\delta s}\|f(s,\cdot)\|_{L^2(\bS^2)}.
\end{align}
Standard interior estimates and boundary estimates using $v(0,\cdot)=0$ give
\begin{align}\label{eq:model_interior_schauder}
	\|v\|_{C^{2,\alpha}([r,r+1]\times\bS^2)}
	\leq C\left(
	\|L_1v\|_{C^{0,\alpha}([r-1,r+2]\times\bS^2)}
	+\|v\|_{L^2([r-1,r+2]\times\bS^2)}
	\right),\quad r\geq 1,
\end{align}
\begin{align}\label{eq:model_boundary_schauder}
	\|v\|_{C^{2,\alpha}([0,2]\times\bS^2)}
	\leq C\left(
	\|L_1v\|_{C^{0,\alpha}([0,4]\times\bS^2)}
	+\|v\|_{L^2([0,4]\times\bS^2)}
	\right).
\end{align}
Here $C$ depends only on $\alpha$. Combining
\eqref{eq:model_weighted_l2}--\eqref{eq:model_boundary_schauder}, we have
\begin{align}\label{eq:model_inverse_estimate}
	\|v\|_{X_\delta}\leq C_\delta\|L_1v\|_{Y_\delta},
	\quad v\in X_\delta.
\end{align}

The same construction proves surjectivity of $L_1$. Given $f\in Y_\delta$, let $f_{j,m}$ be its spherical harmonic coefficients and define $v_{j,m}$ by \eqref{eq:model_green_formula}. The argument leading to \eqref{eq:model_weighted_l2} shows that these coefficients define an even function with zero spherical mean
$v\in L^2_{\mathrm{loc}}([0,\infty)\times\bS^2)$
satisfying
\begin{align}\label{eq:model_constructed_l2}
	\sup_{s\geq0}e^{\delta s}\|v(s,\cdot)\|_{L^2(\bS^2)}
	\leq C_\delta\|f\|_{Y_\delta}.
\end{align}
Let $P_N$ denote the projection onto spherical harmonics of degree at most $N$, and set
$$v_N=\sum_{j=0}^N\sum_{m=0}^{2j} v_{j,m}Y_{j,m},\quad f_N=P_Nf.$$
Then
$$L_1v_N=f_N,\quad v_N(0,\cdot)=0,$$
and $v_N\rightarrow v$, $f_N\rightarrow f$ in $L^2_{\mathrm{loc}}$. The interior and boundary $W^{2,2}$ estimates give uniform local $W^{2,2}$ bounds for $v_N$. Passing to a subsequence and using continuity of the trace map, we obtain
$$L_1v=f,\quad v(0,\cdot)=0.$$
Standard elliptic regularity then gives $v\in C^{2,\alpha}([0,R]\times\bS^2)$ for every $R>0$. Finally, \eqref{eq:model_constructed_l2} and the interior and boundary estimates \eqref{eq:model_interior_schauder}\eqref{eq:model_boundary_schauder} imply
$$v\in X_\delta,\quad \|v\|_{X_\delta}\leq C_\delta\|f\|_{Y_\delta}.$$
Thus $L_1$ is surjective.

	Now set
	$$u_\theta=1+\theta(u-1),\quad 0\leq\theta\leq1.$$
	There is $c>0$, independent of $\theta$, such that $c\leq u_\theta\leq c^{-1}$, and $u_\theta, \, u_\theta^{-1}\rightarrow1$ in $C^{2,\alpha}$ on unit cylinders as $s\rightarrow\infty$, uniformly in $\theta$.
	The injectivity argument applies to every $L_{u_\theta}$.
	
	We claim that there are $R>0$ and $C>0$, independent of $\theta$, such that every $v\in X_\delta$ satisfies
	\begin{align}\label{eq:linear_compact_remainder}
		\|v\|_{X_\delta}\leq C\left(\|L_{u_\theta} v\|_{Y_\delta}+\|v\|_{C^0([0,R]\times\bS^2)}\right).
	\end{align}
	Choose a smooth cutoff $\chi=\chi(s)$ with $\chi=0$ on $[0,R]$, $\chi=1$ on $[R+1,\infty)$, and derivatives bounded independently of $R$. Then
	$$
	L_1(\chi v)=\chi L_{u_\theta}v+2\chi'v_s+(\chi''-3\chi')v+\chi\Delta\bigl((1-u_\theta^{-1})v\bigr).
	$$
	Applying \eqref{eq:model_inverse_estimate} to $\chi v\in X_\delta$ gives
	$$\|\chi v\|_{X_\delta}\leq C\|L_{u_\theta} v\|_{Y_\delta}+C_R\|v\|_{C^{1,\alpha}([R,R+1]\times\bS^2)}+C\varepsilon_R\|v\|_{X_\delta},$$
	where $\varepsilon_R\rightarrow0$ uniformly in $\theta$, $C$ is independent of $R$ and $\theta$, and $C_R$ is independent of $\theta$. On the compact cylinder, the interior and boundary Schauder estimates give
	$$\|v\|_{C^{2,\alpha}([0,R+1]\times\bS^2)}\leq C_R\left(\|L_{u_\theta} v\|_{C^{0,\alpha}([0,R+2]\times\bS^2)}+\|v\|_{C^0([0,R+2]\times\bS^2)}\right).$$
	Combining these estimates using $v=\chi v+(1-\chi)v$, then taking $R$ large enough to absorb $C\varepsilon_R\|v\|_{X_\delta}$, proves \eqref{eq:linear_compact_remainder} after increasing $R$.
	
	We next claim that
	\begin{align}\label{eq:uniform_linear_estimate}
		\|v\|_{X_\delta}\leq C\|L_{u_\theta} v\|_{Y_\delta}
	\end{align}
	for every $v\in X_\delta$, uniformly for $\theta\in[0,1]$. Otherwise there exist $\theta_i\rightarrow\theta_\infty$ and $v_i\in X_\delta$ such that
	$$\|v_i\|_{X_\delta}=1,\quad \|L_{u_{\theta_i}}v_i\|_{Y_\delta}\rightarrow0.$$
	By Lemma \ref{lem:weighted_local_compactness}, after passing to a subsequence, $v_i\rightarrow v_\infty$ locally in $C^{2,\beta}$ for every $0<\beta<\alpha$, with $v_\infty\in X_\delta$. Passing to the limit gives $L_{u_{\theta_\infty}}v_\infty=0$, hence $v_\infty=0$ by injectivity. Therefore $\|v_i\|_{C^0([0,R]\times\bS^2)}\rightarrow0$, contradicting \eqref{eq:linear_compact_remainder}. Thus \eqref{eq:uniform_linear_estimate} holds.
	
	Finally, we prove surjectivity for $L_u$ by the method of continuity. Let
	$$\mathcal{I}_0=\{\theta\in[0,1]:L_{u_\theta}:X_\delta\rightarrow Y_\delta\text{ is an isomorphism}\}.$$
	We have $0\in\mathcal{I}_0$. Since $L_{u_\theta}$ depends continuously on $\theta$ in operator norm, $\mathcal{I}_0$ is open. To prove closedness, let $\theta_i\in\mathcal{I}_0$ with $\theta_i\rightarrow\theta_\infty$, fix $f\in Y_\delta$, and solve
	$$L_{u_{\theta_i}}v_i=f.$$
	By \eqref{eq:uniform_linear_estimate}, $\|v_i\|_{X_\delta}\leq C\|f\|_{Y_\delta}$. After passing to a subsequence, Lemma \ref{lem:weighted_local_compactness} gives $v_i\rightarrow v_\infty$ locally in $C^{2,\beta}$, with $v_\infty\in X_\delta$. Passing to the limit in the equation yields
	$$L_{u_{\theta_\infty}}v_\infty=f.$$
	Thus $L_{u_{\theta_\infty}}$ is surjective, while injectivity was already proved. Hence $\mathcal{I}_0$ is closed. Therefore $\mathcal{I}_0=[0,1]$, and in particular $L_u=L_{u_\theta}$ at $\theta=1$ is an isomorphism.
\end{proof}

\subsection{The continuity method}\label{subsec:continuity_method}

Given a positive even function $\varphi\in C^{2,\alpha}(\bS^2)$ with $\oint_{\bS^2}\varphi\,d\mu=1$, set $\varphi_\tau=1-\tau+\tau\varphi$ for $\tau\in[0,1]$. Then $\varphi_0=1$, $\varphi_1=\varphi$, and
\begin{align}\label{eq:path_bounds}
m_0\leq\varphi_\tau\leq M_0,\quad\oint_{\bS^2}\varphi_\tau\,d\mu=1,\quad0\leq\oint_{\bS^2}\varphi_\tau\log\varphi_\tau\,d\mu\leq\tau \Ent_0\leq \Ent_0.
\end{align}
where we used the convexity of the entropy. Thus the boundary extrema and entropy are uniformly controlled along this path.

Define
\begin{align}\label{eq:boundary_extension}
E_\tau(s,x)=1+\tau e^{-s}(\varphi(x)-1),
\end{align}
so $E_\tau(0,\cdot)=\varphi_\tau$ and $\oint_{\bS^2}E_\tau(s,x)d\mu=1$.

Fix $\Lambda=C_*M_0^4$ and choose
\begin{align}\label{eq:weight_choice}
0<\delta<\min\{1,\beta_\Lambda/11\}.
\end{align}
By Theorem \ref{thm:c0_estimate}, Corollary \ref{cor:derivative_decay}, and \eqref{eq:path_bounds}, every positive even solution $u$ of \eqref{eq:s_pde} and \eqref{eq:boundary_asymptotics} with boundary value $\varphi_\tau$ satisfies $\|u-1\|_{C_\delta^{2,\alpha}}\leq C$, where $C$ is independent of $\tau$ and $u$. The bound near $s=0$ also uses the uniform $C^{2,\alpha}$ norm of $\varphi_\tau$. Since $\delta<1$, \eqref{eq:boundary_extension} also gives $\|E_\tau-1\|_{C_\delta^{2,\alpha}}\leq C$. Thus $u-E_\tau\in X_\delta$, and 
$\|u-E_\tau\|_{X_\delta}\leq\|u-1\|_{C_\delta^{2,\alpha}}+\|E_\tau-1\|_{C_\delta^{2,\alpha}}\leq C.$

Let
\begin{align}\label{eq:continuity_set}
\mathcal{I}=\{\tau\in[0,1]|\,\mathcal{P}(u)=0\text{ for some }u>0\text{ with }u-E_\tau\in X_\delta\}.
\end{align}
The condition $u-E_\tau\in X_\delta$ gives the required boundary value, evenness, unit spherical mean, and uniform convergence to $1$ for $u$.

We show that $\mathcal{I}=[0,1]$.

\begin{itemize}
\item $\mathcal{I}$ is non-empty: $0\in\mathcal{I}$, since $u\equiv1$ is a solution.

\item $\mathcal{I}$ is closed: Let $\tau_i\in\mathcal{I}$ with $\tau_i\rightarrow\tau_\infty\in[0,1]$, and let $u_{\tau_i}$ be corresponding solutions with $h_i=u_{\tau_i}-E_{\tau_i}\in X_\delta$. By the uniform $X_\delta$ bound established after \eqref{eq:weight_choice} and Lemma \ref{lem:weighted_local_compactness}, we may pass to a subsequence such that $h_i\rightarrow h_\infty\in X_\delta$ locally in $C^{2,\beta}$ for every $0<\beta<\alpha$. Set $u_\infty=E_{\tau_\infty}+h_\infty$. Then $u_{\tau_i}\rightarrow u_\infty$ locally in $C^{2,\beta}$, and $u_\infty>0$ by the uniform lower bound from \eqref{eq:global_c0_estimate} and \eqref{eq:path_bounds}. Passing to the limit in $\mathcal{P}(u_{\tau_i})=0$ gives $\mathcal{P}(u_\infty)=0$. Thus $\tau_\infty\in\mathcal{I}$.

\item $\mathcal{I}$ is open: We use the implicit function theorem. Assume $\tau_0\in\mathcal{I}$ and let $u_{\tau_0}$ be a corresponding solution. For $(\tau,h)\in\bR\times X_\delta$ with $E_\tau+h>0$, define
\begin{align}\label{eq:implicit_function_map}
	\mathcal{F}(\tau,h)=\mathcal{P}(E_\tau+h)\in Y_\delta.
\end{align}
Then $\mathcal{F}$ is smooth on its domain.
 Set $h_0=u_{\tau_0}-E_{\tau_0}\in X_\delta$. Then $\mathcal{F}(\tau_0,h_0)=0$, and
$$D_h\mathcal{F}(\tau_0,h_0)=L_{u_{\tau_0}}:X_\delta\rightarrow Y_\delta$$
is an isomorphism by Lemma \ref{lem:linear_isomorphism} and \eqref{eq:weight_choice}. By the implicit function theorem, for every $\tau'\in[0,1]$ sufficiently close to $\tau_0$, there is an $h\in X_\delta$ such that $E_{\tau'}+h>0$ and $\mathcal{P}(E_{\tau'}+h)=0$. Thus $\mathcal{I}$ is open in $[0,1]$.

\end{itemize}

Since $\mathcal{I}$ is non-empty, open, and closed in $[0,1]$, we have $\mathcal{I}=[0,1]$. Taking $\tau=1$ gives a solution with boundary value $\varphi$. Uniqueness follows from \eqref{eq:ball_transform} and Corollary \ref{cor:uniqueness}, and the regularity assertions follow from interior and boundary elliptic estimates. This proves Theorem \ref{thm:existence_uniqueness}.

\section{The general case on the infinite cylinder}\label{sec:general_case}

We now consider positive solutions of \eqref{eq:s_pde} without the evenness assumption. A stationary solution $u(s,x)=h(x)>0$ satisfies
\begin{align}\label{eq:general_stationary_equation}
	\Delta\log h+2h=2.
\end{align}
The smooth positive solutions of \eqref{eq:general_stationary_equation} are precisely the conformal Jacobians
\begin{align}\label{eq:general_stationary_family}
	\mathfrak{j}_a(x)=\frac{1}{(\sqrt{1+|a|^2}-a\cdot x)^2},\quad a\in\bR^3.
\end{align}
Equivalently, $h=J_\phi$ for some conformal transformation $\phi:\bS^2\rightarrow\bS^2$, where $J_\phi$ is its Jacobian.

For the general boundary problem, we allow
\begin{align}\label{eq:general_boundary_asymptotics}
	u(0,x)=\varphi(x),\quad \lim\limits_{s\rightarrow\infty}\sup_{x\in\bS^2}|u(s,x)-h(x)|=0,
\end{align}
where $h$ is one of these stationary solutions. The limiting function $h$ is not prescribed. We assume that $u$ is continuous up to $s=0$ and smooth in the interior. Integrating \eqref{eq:general_stationary_equation} gives $\oint_{\bS^2}h\,d\mu=1$. The spherical mean of $u$ satisfies the same ODE as in the derivation of \eqref{eq:mass_one}, so \eqref{eq:general_boundary_asymptotics} again implies \eqref{eq:mass_one}.

Our goal is the following extension of Theorem \ref{thm:existence_uniqueness}.

\begin{theorem}\label{thm:general_existence_uniqueness}
	For any positive $\varphi\in C^{2,\alpha}(\bS^2)$, $0<\alpha<1$, satisfying $\oint_{\bS^2}\varphi\,d\mu=1$, there is a unique positive solution of \eqref{eq:s_pde} and \eqref{eq:general_boundary_asymptotics}, with $h$ satisfying \eqref{eq:general_stationary_equation}. In particular, $h$ is uniquely determined by $\varphi$.
\end{theorem}

The uniqueness argument in Corollary \ref{cor:uniqueness} applies through \eqref{eq:ball_transform}: the corresponding functions $w_i$ have boundary value $\log(\varphi/4)$ at $\xi=0$ and satisfy \eqref{eq:intro_degenerate_end}, since the $u_i$ are bounded. Thus $u_1=u_2$, and hence $h_1=h_2$, even if their limiting functions were not assumed to agree.

We first study the conformally invariant deficit between entropy and $H^{-1}$ energy on $\bS^2$. The inequalities established here will be applied to each spherical slice of a solution to prove decay of the deficit in Section \ref{subsec:general_pde_estimates}.

\subsection{Functionals and inequalities on $\bS^2$}\label{subsec:general_sphere_functionals}

For a smooth mean-zero function $g$ on $\bS^2$, let $v$ solve
$$-\Delta v=g,\quad \oint_{\bS^2}v\,d\mu=0.$$
Define
$$\|g\|_{H^{-1}(\bS^2)}^2:=\oint_{\bS^2}|\nabla v|^2\,d\mu=\oint_{\bS^2}vg\,d\mu$$

Let $u>0$ be a smooth function on $\bS^2$ with $\oint_{\bS^2}u\,d\mu=1$, and let $f$ be the unique mean-zero solution of
\begin{align}\label{eq:general_potential}
	-\Delta f=u-1,\quad \oint_{\bS^2}f\,d\mu=0.
\end{align}
Define the entropy and the $H^{-1}$ energy by
\begin{align}\label{eq:general_entropy_energy}
	\Ent(u)=\oint_{\bS^2}u\log u\,d\mu,\quad
	E(u)=\|u-1\|_{H^{-1}(\bS^2)}^2=\oint_{\bS^2}|\nabla f|^2\,d\mu.
\end{align}
The deficit between entropy and $H^{-1}$ energy is
\begin{align}\label{eq:general_deficit}
	D(u)=\Ent(u)-E(u).
\end{align}
Set
\begin{align}\label{eq:general_q_operator}
	\mathcal{Q}(u)=\log u-2f,\quad Q(u)=\oint_{\bS^2}|\nabla\mathcal{Q}(u)|^2\,d\mu.
\end{align}
These definitions give the identity
\begin{align}\label{eq:general_liouville_identity}
	Q(u)=\oint_{\bS^2}|\nabla\log u|^2\,d\mu+4\oint_{\bS^2}\log u\,d\mu-4D(u),
\end{align}
where the sum of the first two terms is known as the Liouville energy. $Q(u)$ is also the $H^{-1}$ energy of $2u(K_{ug_{\bS^2}}-1)$.

Let $\phi:\bS^2\rightarrow\bS^2$ be a conformal transformation, and let $J_\phi$ be its Jacobian, so $\phi^*g_{\bS^2}=J_\phi g_{\bS^2}.$ Then
\begin{align}\label{eq:general_conformal_jacobian}
	\Delta\log J_\phi+2J_\phi=2,\quad \oint_{\bS^2}J_\phi\,d\mu=1.
\end{align}
Define
\begin{align}\label{eq:general_conformal_action}
	u_\phi=J_\phi(u\circ\phi).
\end{align}
The potential corresponding to $u_\phi$ is
\begin{align}\label{eq:general_transformed_potential}
	f_\phi=f\circ\phi+\frac12\log J_\phi+c_\phi,
\end{align}
where $c_\phi$ is constant and chosen so that $\oint_{\bS^2}f_\phi\,d\mu=0$. Indeed, \eqref{eq:general_conformal_jacobian} gives $-\Delta f_\phi=u_\phi-1$.

The entropy $\Ent(u)$ and the $H^{-1}$ energy $E(u)$ are not separately conformally invariant, but their difference $D(u)$ is. The functional $Q(u)$ is also conformally invariant. The nonnegativity of $D(u)$ and the characterization of equality in the following lemma are well-known consequences of the Onofri inequality by duality. We include a proof for completeness. We also verify the conformal invariance of $D$ and $Q$.

\begin{lemma}\label{lem:general_deficit_invariance}
	For every smooth positive function $u$ on $\bS^2$ with $\oint_{\bS^2}u\,d\mu=1$, we have $D(u)\geq0$, with equality if and only if $u=J_\phi$ for some conformal transformation $\phi$ of $\bS^2$. Moreover, for every conformal transformation $\phi$,
	\begin{align}\label{eq:general_deficit_invariance}
		D(u_\phi)=D(u),\quad Q(u_\phi)=Q(u).
	\end{align}
\end{lemma}

\begin{proof}
	The variational formula for entropy is
	\begin{align}\label{eq:general_entropy_variational}
		\Ent(u)=\sup_{g\in H^1(\bS^2)}\left(\oint_{\bS^2}ug\,d\mu-\log\left(\oint_{\bS^2}e^g\,d\mu\right)\right).
	\end{align}
	Taking $g=2f$ and applying the ordinary sharp Onofri inequality \cite[equation (4)]{OsgoodPhillipsSarnak1988}, we obtain
	\begin{align}
		\Ent(u)&\geq2\oint_{\bS^2}uf\,d\mu-\log\left(\oint_{\bS^2}e^{2f}\,d\mu\right)\\
		&\geq2\oint_{\bS^2}uf\,d\mu-2\oint_{\bS^2}f\,d\mu-\oint_{\bS^2}|\nabla f|^2\,d\mu\\
		&=2E(u)-E(u)=E(u).
	\end{align}
	Here we used $\oint_{\bS^2}f\,d\mu=0$ and $\oint_{\bS^2}uf\,d\mu=E(u)$. Thus $D(u)\geq0$.
	If $D(u)=0$, both inequalities are equalities, so $g=2f$ attains the supremum in \eqref{eq:general_entropy_variational}. Hence
	$$u=\frac{e^{2f}}{\oint_{\bS^2}e^{2f}\,d\mu},\quad \Delta\log u+2u=2,$$
	where the second identity follows from \eqref{eq:general_potential}. By the classification of solutions of \eqref{eq:general_stationary_equation}, we have $u=J_\phi$ for some conformal transformation $\phi$.

	For any conformal transformation $\phi$, set $w=\log J_\phi$. Equality in the same Onofri inequality for $w$ gives
	$$\oint_{\bS^2}w\,d\mu=-\frac14\oint_{\bS^2}|\nabla w|^2\,d\mu.$$
	Together with \eqref{eq:general_conformal_jacobian}, this yields
	\begin{align}\label{eq:general_jacobian_energy}
		\oint_{\bS^2}J_\phi w\,d\mu
		=\frac12\oint_{\bS^2}|\nabla w|^2\,d\mu+\oint_{\bS^2}w\,d\mu
		=\frac14\oint_{\bS^2}|\nabla w|^2\,d\mu.
	\end{align}
	By change of variables,
	$$\Ent(u_\phi)-\Ent(u)=\oint_{\bS^2}u_\phi w\,d\mu.$$
	Conformal invariance of the Dirichlet integral gives $\oint_{\bS^2}|\nabla(f\circ\phi)|^2\,d\mu=E(u)$. Using \eqref{eq:general_transformed_potential} to expand $|\nabla f_\phi|^2$ and integrating by parts, we obtain
	\begin{align}
		E(u_\phi)-E(u)
		&=\oint_{\bS^2}\nabla(f\circ\phi)\cdot\nabla w\,d\mu+\frac14\oint_{\bS^2}|\nabla w|^2\,d\mu\\
		&=-\oint_{\bS^2}w\Delta(f\circ\phi)\,d\mu+\frac14\oint_{\bS^2}|\nabla w|^2\,d\mu\\
		&=\oint_{\bS^2}(u_\phi-J_\phi)w\,d\mu+\frac14\oint_{\bS^2}|\nabla w|^2\,d\mu\\
		&=\oint_{\bS^2}u_\phi w\,d\mu.
	\end{align}
	Here we used $-\Delta(f\circ\phi)=u_\phi-J_\phi$, and the last equality follows from \eqref{eq:general_jacobian_energy}.
	Hence $D(u_\phi)=D(u)$. 

	Finally, \eqref{eq:general_transformed_potential} gives
	$$\mathcal{Q}(u_\phi)=\mathcal{Q}(u)\circ\phi-2c_\phi.$$
	Conformal invariance of the Dirichlet integral therefore gives $Q(u_\phi)=Q(u)$.
\end{proof}

Nonconstant conformal Jacobians satisfy $\Ent(J_\phi)=E(J_\phi)>0$ by Lemma \ref{lem:general_deficit_invariance}, so an estimate $\Ent(u)\geq(1+\epsilon)E(u)$ with a fixed $\epsilon>0$ cannot hold in general. However, if $e^{2f}$ is balanced, we can obtain such an improvement using the following improved Onofri inequality of Chang--Yang \cite[Proposition B]{ChangYang1987}. Gui--Moradifam \cite[Theorem 1.2]{GuiMoradifam2018} proved the sharp version of \eqref{eq:chang_yang_onofri} with coefficient $a=1/8$, but the weaker inequality suffices for the compactness argument in Lemma \ref{lem:general_deficit_coercivity}.

\begin{lemma}[Improved Onofri inequality]\label{lem:chang_yang_onofri}
	There is a universal constant $a\in(\frac18,\frac14)$ such that every $\psi\in H^1(\bS^2)$ with $e^\psi$ balanced, that is,
	$$\oint_{\bS^2}e^\psi x_i\,d\mu=0,\quad i=1,2,3,$$
	satisfies
	\begin{align}\label{eq:chang_yang_onofri}
		\log\left(\oint_{\bS^2}e^\psi\,d\mu\right)\leq\oint_{\bS^2}\psi\,d\mu+a\oint_{\bS^2}|\nabla\psi|^2\,d\mu.
	\end{align}
\end{lemma}

\begin{lemma}\label{lem:balanced_potential_entropy}
	Let $u>0$ be a smooth function on $\bS^2$ with $\oint_{\bS^2}u\,d\mu=1$, and let $f$ be the mean-zero potential defined by \eqref{eq:general_potential}. Suppose $e^{2f}$ is balanced, that is,
	$$\oint_{\bS^2}e^{2f}x_i\,d\mu=0,\quad i=1,2,3.$$
	Then, with $a$ as in Lemma \ref{lem:chang_yang_onofri},
	\begin{align}\label{eq:balanced_potential_entropy}
		\Ent(u)\geq(2-4a)\|u-1\|_{H^{-1}(\bS^2)}^2,\quad D(u)\geq(1-4a)\|u-1\|_{H^{-1}(\bS^2)}^2.
	\end{align}
\end{lemma}

\begin{proof}
	The variational formula \eqref{eq:general_entropy_variational} and Lemma \ref{lem:chang_yang_onofri}, applied to $2f$, give
	\begin{align}
		\Ent(u)
		&\geq2\oint_{\bS^2}uf\,d\mu-\log\left(\oint_{\bS^2}e^{2f}\,d\mu\right)\\
		&\geq2\oint_{\bS^2}uf\,d\mu-4a\oint_{\bS^2}|\nabla f|^2\,d\mu-2\oint_{\bS^2}f\,d\mu\\
		&=(2-4a)\oint_{\bS^2}|\nabla f|^2\,d\mu.
	\end{align}
	The last equality follows from \eqref{eq:general_potential} by integration by parts. The bound for $D(u)$ follows from the definition \eqref{eq:general_deficit}.
\end{proof}

The next lemma bounds $Q(u)$ away from zero when $D(u)$ is bounded above and away from zero, while Lemma \ref{lem:general_small_deficit_coercivity} gives a lower bound proportional to $D(u)$ when the deficit is small. These estimates will be used to prove the decay of the deficit along solutions in Lemma \ref{lem:general_deficit_decay}.

\begin{lemma}\label{lem:general_deficit_coercivity}
	For every $0<\epsilon\leq M$, there is a constant $\kappa=\kappa(\epsilon,M)>0$ such that every smooth positive function $u$ on $\bS^2$ satisfying
	$$\oint_{\bS^2}u\,d\mu=1,\quad \epsilon\leq D(u)\leq M,$$
	has $Q(u)\geq\kappa$.
\end{lemma}

\begin{proof}
	Suppose the conclusion fails. There is a sequence of smooth positive functions $u_k$ with $\oint_{\bS^2}u_k\,d\mu=1$, $\epsilon\leq D(u_k)\leq M$, and $Q(u_k)\rightarrow0$. Let $f_k$ be the mean-zero potential in \eqref{eq:general_potential}. By Hersch's conformal balancing \cite{Hersch1970}, and \eqref{eq:general_transformed_potential}, we may put each $u_k$ in the balanced gauge
	$$\oint_{\bS^2}e^{2f_k}x_i\,d\mu=0,\quad i=1,2,3.$$
	This does not change $D(u_k)$ or $Q(u_k)$, by Lemma \ref{lem:general_deficit_invariance}. Lemma \ref{lem:balanced_potential_entropy} gives
	\begin{align}\label{eq:general_balanced_energy_bound}
		\oint_{\bS^2}|\nabla f_k|^2d\mu&=\|u_k-1\|_{H^{-1}(\bS^2)}^2\leq\frac{D(u_k)}{1-4a}\leq\frac{M}{1-4a}.
	\end{align}	
	Set
	$$g_k=\log u_k-2f_k-\oint_{\bS^2}(\log u_k-2f_k)d\mu.$$
	Since $\oint_{\bS^2}|\nabla g_k|^2d\mu=Q(u_k)\rightarrow0$, the Poincar\'e inequality gives $g_k\rightarrow0$ in $H^1(\bS^2)$. The mean-zero condition and \eqref{eq:general_balanced_energy_bound} also bound $f_k$ in $H^1(\bS^2)$. The Moser--Trudinger inequality \cite{Moser1971} therefore bounds $e^{g_k+2f_k}$ in $L^q(\bS^2)$ for every finite $q\geq1$. Since $g_k+2f_k$ has mean zero, Jensen's inequality gives $\oint_{\bS^2}e^{g_k+2f_k}d\mu\geq1$. Thus
	\begin{align}\label{eq:general_exponential_representation}
		u_k=\frac{e^{g_k+2f_k}}{\oint_{\bS^2}e^{g_k+2f_k}d\mu}
	\end{align}
	is uniformly bounded in $L^q(\bS^2)$ for every finite $q\geq1$.
	
	Since 
	\begin{align}\label{eq:k-equation_of_u_k_f_k}
		-\Delta f_k=u_k-1 \,\, \text{and} \,\, \oint_{\bS^2}f_k\,d\mu=0,
	\end{align}
	elliptic estimates bound $f_k$ in $W^{2,q}(\bS^2)$ for every finite $q>1$. Note also that $g_k\rightarrow0$ in $H^1(\bS^2)$, after passing to a subsequence, $f_k\rightarrow f_\infty$ in $C^{1,\alpha}(\bS^2)$ for any fixed $0<\alpha<1$, and $g_k\rightarrow0$ almost everywhere. The uniform exponential bounds imply, by uniform integrability, that $e^{g_k+2f_k}\rightarrow e^{2f_\infty}$ in every $L^q$. Therefore \eqref{eq:general_exponential_representation} gives
	$$u_k\rightarrow u_\infty=\frac{e^{2f_\infty}}{\oint_{\bS^2}e^{2f_\infty}d\mu}>0\quad\mathrm{in}\quad L^q(\bS^2),\quad 1\leq q<\infty.$$
	Taking logarithms gives
	$\log u_\infty=2f_\infty+c_\infty,\,\, c_\infty=-\log\left(\oint_{\bS^2}e^{2f_\infty}d\mu\right).$
	Passing \eqref{eq:k-equation_of_u_k_f_k} to the limit gives $-\Delta f_\infty=u_\infty-1$. Hence,
	$$\Delta\log u_\infty=-2(u_\infty-1).$$
	Elliptic regularity implies $u_\infty\in C^\infty(\bS^2)$, so $u_\infty=J_\phi$ for some conformal transformation $\phi$. Lemma \ref{lem:general_deficit_invariance} gives $D(u_\infty)=D(1)=0$.
	
	Finally, the strong $L^q$ convergence of $u_k$ for $q>1$ gives $\Ent(u_k)\rightarrow\Ent(u_\infty)$, while $f_k\rightarrow f_\infty$ in $H^1$. Thus  $D(u_k)\rightarrow D(u_\infty)=0$, contradicting $D(u_k)\geq\epsilon$.
\end{proof}

\begin{lemma}\label{lem:general_small_deficit_coercivity}
There exists a universal constant $c_0>0$ such that, for every $\epsilon_0\in(0,8)$, every smooth positive function $u$ on $\bS^2$ with $\oint_{\bS^2}u\,d\mu=1$ and $0\leq D(u)\leq c_0\epsilon_0^2$ satisfies
	\begin{align}\label{eq:general_small_deficit_coercivity}
		Q(u)\geq(8-\epsilon_0)D(u).
	\end{align}
\end{lemma}

\begin{proof}
	Suppose the conclusion fails. There are $\epsilon_k\in(0,8)$ and smooth positive functions $u_k$ with $\oint_{\bS^2}u_k\,d\mu=1$, $0<d_k:=D(u_k)\leq\epsilon_k^2/k$, and $Q(u_k)<(8-\epsilon_k)d_k$. In particular, $d_k\rightarrow0$.
	Using the conformal balancing in the proof of Lemma \ref{lem:general_deficit_coercivity}, we may assume
	$$\|\nabla f_k\|_{L^2(\bS^2)}^2=\|u_k-1\|_{H^{-1}(\bS^2)}^2\leq\frac{d_k}{1-4a}.$$
	Set $v_k=\log u_k-\oint_{\bS^2}\log u_k\,d\mu$. Since $Q(u_k)<8d_k$, the Poincar\'e inequality implies
	$$\|v_k\|_{H^1(\bS^2)}\leq C\|\nabla v_k\|_{L^2(\bS^2)}\leq C(\sqrt{Q(u_k)}+2\|\nabla f_k\|_{L^2(\bS^2)})\leq Cd_k^{1/2}.$$
	We now express $D(u_k)$ and $Q(u_k)$ in terms of $v_k$, up to errors of order $d_k^{3/2}$.
	The Sobolev inequality and the Moser--Trudinger inequalities give $\|v_k\|_{L^p(\bS^2)}\leq C_p d_k^{1/2}$ and $\oint_{\bS^2}e^{p|v_k|}\,d\mu\leq C_p$ for every fixed $1\leq p<\infty$. Thus, H\"older's inequality gives
	$$\oint_{\bS^2}|v_k|^p e^{4|v_k|}\,d\mu\leq\left(\oint_{\bS^2}|v_k|^{2p}\,d\mu\right)^{1/2}\left(\oint_{\bS^2}e^{8|v_k|}\,d\mu\right)^{1/2}\leq C_p d_k^{p/2}.$$
	Taylor's remainder theorem gives
	\begin{align}
		|e^{v_k}-1-v_k-\frac12v_k^2|\leq\frac16|v_k|^3e^{|v_k|}.
	\end{align}
	Hence,
	\begin{align}\label{eq:general_small_deficit_normalization}
		c_k:=\oint_{\bS^2}e^{v_k}\,d\mu=1+\frac12\oint_{\bS^2}v_k^2\,d\mu+O(d_k^{3/2}).
	\end{align}
	Note that $u_k=\frac{e^{v_k}}{c_k}$, hence
	\begin{align}\label{eq:general_small_deficit_mean_log}
		\oint_{\bS^2}\log u_k\,d\mu=\oint_{\bS^2}(v_k-\log c_k)\,d\mu=-\log c_k=-\frac12\oint_{\bS^2}v_k^2\,d\mu+O(d_k^{3/2}),
	\end{align}
	\begin{align}
		\Ent(u_k)=\oint_{\bS^2}\frac{e^{v_k}}{c_k}\log\left(\frac{e^{v_k}}{c_k}\right)\,d\mu=\frac12\oint_{\bS^2}v_k^2\,d\mu+O(d_k^{3/2}).
	\end{align}
	\begin{align}
		\|u_k-1-v_k\|_{H^{-1}(\bS^2)}^2\leq\frac12\|u_k-1-v_k\|_{L^2(\bS^2)}^2=O(d_k^2).\label{eq:general_small_deficit_taylor}
	\end{align}
	Thus,
	\begin{align}
		D(u_k)&=\frac12\oint_{\bS^2}v_k^2\,d\mu-\|v_k\|_{H^{-1}(\bS^2)}^2+O(d_k^{3/2}),\label{eq:general_small_deficit_D_expansion}\\
		Q(u_k)&=\oint_{\bS^2}|\nabla\log u_k|^2\,d\mu+4\oint_{\bS^2}\log u_k\,d\mu-4D(u_k)\\
		&=\oint_{\bS^2}|\nabla v_k|^2\,d\mu-4\oint_{\bS^2}v_k^2\,d\mu+4\|v_k\|_{H^{-1}(\bS^2)}^2+O(d_k^{3/2}).\label{eq:general_small_deficit_Q_expansion}
	\end{align}
	
	The nonzero eigenvalues of $-\Delta$ on $\bS^2$ are $\lambda_j=j(j+1)$, $j\geq1$, and
	\begin{align}\label{eq:general_small_deficit_spectral}
		\lambda_j-4+\frac{4}{\lambda_j}\geq8(\frac12-\frac{1}{\lambda_j}),\quad j\geq1.
	\end{align}
	Expanding the mean-zero function $v_k$ in spherical harmonics and using \eqref{eq:general_small_deficit_D_expansion}--\eqref{eq:general_small_deficit_spectral}, we conclude that
	$$Q(u_k)\geq8d_k-Cd_k^{3/2}=(8-C\sqrt{d_k})d_k,$$
	where $C$ is universal. Since $C\sqrt{d_k}\leq C\epsilon_k/\sqrt{k}<\epsilon_k$ for large $k$, this contradicts $Q(u_k)<(8-\epsilon_k)d_k$.
\end{proof}

\subsection{A priori estimates}\label{subsec:general_pde_estimates}

Throughout this subsection, we assume that $u>0$ solves \eqref{eq:s_pde} on $(0,\infty)\times\bS^2$, is continuous up to $s=0$, and is smooth in the interior. We impose the boundary and asymptotic conditions \eqref{eq:boundary_asymptotics}:
$$u(0,x)=\varphi(x)>0,\quad \lim\limits_{s\rightarrow\infty}\sup_{x\in\bS^2}|u(s,x)-1|=0.$$
Note that the limit condition gives restrictions to the boundary value $\varphi$.

The equation is equivariant under the conformal action \eqref{eq:general_conformal_action}, applied with $\phi$ independent of $s$. More precisely, for $\mathcal{P}$ defined in \eqref{eq:nonlinear_operator},
\begin{align}\label{eq:general_conformal_covariance}
	\mathcal{P}(u_\phi)(s,x)=J_\phi(x)\mathcal{P}(u)(s,\phi(x)).
\end{align}
For a solution with limiting function $h$ as in \eqref{eq:general_boundary_asymptotics}, the transformed solution has boundary value $\varphi_\phi=J_\phi(\varphi\circ\phi)$ and limit $h_\phi=J_\phi(h\circ\phi)$. Since $h$ is a conformal Jacobian, we may choose the fixed map $\phi$ to be the inverse of the corresponding conformal transformation, giving $h_\phi=1$. Thus the normalized problem considered here is equivalent, after a fixed conformal change and the corresponding change of boundary data, to the version at the beginning of Section \ref{sec:general_case}. We again write $u$ and $\varphi$ for the normalized solution and boundary value; all bounds below refer to this normalized boundary value.

By \eqref{eq:mass_one}, the functionals in Section \ref{subsec:general_sphere_functionals} apply to each slice. Let $f(s,\cdot)$ be the mean-zero potential of $u(s,\cdot)$, and write
\begin{align}\label{eq:general_slice_functionals}
	\Ent(s)&=\Ent(u(s,\cdot)),\quad E(s)=E(u(s,\cdot)),\\
	D(s)&=D(u(s,\cdot)),\quad Q(s)=Q(u(s,\cdot)).
\end{align}
We retain $F_q(s)$ from \eqref{eq:entropy_moments}. Set $q(s,x)=\mathcal{Q}(u(s,\cdot))(x)=\log u(s,x)-2f(s,x)$. Then \eqref{eq:s_pde} becomes
\begin{align}\label{eq:general_q_equation}
	u_{ss}-3u_s=-\Delta q.
\end{align}
We also set
\begin{align}\label{eq:general_kinetic_energy}
	I(s)=\|u_s(s,\cdot)\|_{H^{-1}(\bS^2)}^2
	=\oint_{\bS^2}|\nabla f_s|^2\,d\mu
	=\oint_{\bS^2}u_s f_s\,d\mu.
\end{align}

\begin{lemma}\label{lem:general_deficit_identities}
	For $s>0$, we have
	\begin{align}\label{eq:general_deficit_second_order}
		D''-3D'=Q+\oint_{\bS^2}\frac{u_s^2}{u}\,d\mu-2I,
	\end{align}
	and
	\begin{align}\label{eq:general_deficit_first_order}
		D'=\oint_{\bS^2}q u_s\,d\mu,\quad I'=6I+2D'.
	\end{align}
	Moreover,
	\begin{align}\label{eq:general_deficit_kinetic_identity}
		D(s)=\frac12I(s)+3\int_s^\infty I(r)\,dr,\quad I(s)\leq2D(s).
	\end{align}
\end{lemma}

\begin{proof}
	The entropy computation in the proof of Lemma \ref{lem:entropy_ode} gives
	\begin{align}
		\Ent''=\oint_{\bS^2}\left(\frac{u_s^2}{u}+|\nabla\log u|^2\right)\,d\mu+3\Ent'-2\Ent+2\oint_{\bS^2}\log u\,d\mu.
	\end{align}
	Since $E=\oint_{\bS^2}(u-1)f\,d\mu$ and $-\Delta f_s=u_s$, we have
	$$E'=2\oint_{\bS^2}u_s f\,d\mu=2\oint_{\bS^2}u f_s\,d\mu,$$
	and
	\begin{align}
		E''&=2\oint_{\bS^2}u_{ss}f\,d\mu+2\oint_{\bS^2}u_s f_s\,d\mu\\
		&=3E'-4E+2\oint_{\bS^2}(u-1)\log u\,d\mu+2I.
	\end{align}
	Subtracting these identities and using $D=\Ent-E$ and \eqref{eq:general_liouville_identity} gives \eqref{eq:general_deficit_second_order}.

	Differentiating $D$ gives $D'=\oint_{\bS^2}q u_s\,d\mu$. By \eqref{eq:general_q_equation},
	\begin{align}
		I'=2\oint_{\bS^2}u_{ss}f_s\,d\mu
		=6I+2\oint_{\bS^2}q u_s\,d\mu=6I+2D'.
	\end{align}
	This proves \eqref{eq:general_deficit_first_order}.

	The uniform convergence $u\rightarrow1$ makes the equation uniformly elliptic on a tail. The interior estimates and interpolation used in the proof of Corollary \ref{cor:derivative_decay} then give $u_s\rightarrow0$ uniformly. Consequently,
	$$D(s)\rightarrow0,\quad I(s)\rightarrow0,\quad D'(s)\rightarrow0.$$
	Integrating $I'-6I=2D'$ from $s$ to infinity gives \eqref{eq:general_deficit_kinetic_identity}.
\end{proof}

\begin{lemma}\label{lem:general_deficit_monotonicity}
	For every $s>0$,
	\begin{align}\label{eq:general_deficit_monotonicity}
		D'(s)+2I(s)\leq0.
	\end{align}
	In particular, $D$ is nonincreasing on $[0,\infty)$.
\end{lemma}

\begin{proof}
	By \eqref{eq:general_q_equation},
	$$Q=\|u_{ss}-3u_s\|_{H^{-1}(\bS^2)}^2.$$
	By \eqref{eq:general_kinetic_energy} and differentiation,
	$$I=\|u_s\|_{H^{-1}(\bS^2)}^2,\quad
	I'=2\langle u_{ss},u_s\rangle_{H^{-1}(\bS^2)}.$$
	Thus \eqref{eq:general_deficit_second_order} gives
	\begin{align}\label{eq:general_deficit_monotonicity_ode}
		(D'+2I)'-3(D'+2I)
		&=Q+2I'-8I+\oint_{\bS^2}\frac{u_s^2}{u}\,d\mu\\
		&=\|u_{ss}-u_s\|_{H^{-1}(\bS^2)}^2+\oint_{\bS^2}\frac{u_s^2}{u}\,d\mu\geq0.
	\end{align}
	Hence $e^{-3s}(D'+2I)$ is nondecreasing. Its limit is zero by the proof of Lemma \ref{lem:general_deficit_identities}, so $D'+2I\leq0$. Since $I\geq0$, this also gives $D'\leq0$.
\end{proof}

\begin{lemma}\label{lem:general_deficit_decay}
	There is a constant $C>0$, depending only on an upper bound of $D(\varphi)$, such that
	\begin{align}\label{eq:general_deficit_decay}
		0\leq D(s)\leq Ce^{-\frac12s},\quad s\geq0.
	\end{align}
\end{lemma}

\begin{proof}
	By definition, $D(0)=D(\varphi)$. If $D(0)=0$, Lemma \ref{lem:general_deficit_monotonicity} and $D\geq0$ give $D\equiv0$. Assume $D(0)>0$. Let $d_*=4c_0$, where $c_0$ is the universal constant in Lemma \ref{lem:general_small_deficit_coercivity}. Thus
	$$Q(s)\geq6D(s)\quad\text{whenever}\quad D(s)\leq d_*.$$
	Set $T=0$ if $D(0)\leq d_*$. Otherwise, let $T$ be the first time at which $D(T)=d_*$; it is finite because $D(s)\rightarrow0$.

	By Lemma \ref{lem:general_deficit_monotonicity}, $D(s)\leq d_*$ for $s\geq T$. Combining \eqref{eq:general_deficit_second_order} and \eqref{eq:general_deficit_kinetic_identity}, we obtain
	\begin{align}\label{eq:general_deficit_small_ode}
		D''-3D'\geq Q-2I\geq6D-4D=2D,\quad s>T.
	\end{align}
	The same ODE comparison as in Corollary \ref{cor:entropy_decay}, using $D(s)\rightarrow0$, gives
	\begin{align}\label{eq:general_deficit_tail_decay}
		D(s)\leq D(T)e^{-\frac{\sqrt{17}-3}{2}(s-T)}
		\leq D(0)e^{-\frac12(s-T)},\quad s\geq T.
	\end{align}

	It remains to bound $T$ in terms of $D(0)$ when $T>0$. On $[0,T]$, we have $d_*\leq D(s)\leq D(0)$, so Lemma \ref{lem:general_deficit_coercivity} gives $Q(s)\geq\kappa$, where $\kappa=\kappa(d_*,D(0))>0$. By \eqref{eq:general_deficit_second_order} and \eqref{eq:general_deficit_monotonicity},
	$$D''-3D'\geq\kappa-2I\geq\kappa+D',\quad
	(e^{-4s}D')'\geq\kappa e^{-4s}.$$
	Integrating from $s$ to $T$ and using $D'(T)\leq0$ yields
	$$D'(s)\leq-\frac{\kappa}{4}\left(1-e^{-4(T-s)}\right).$$
	Integrating once more gives
	\begin{align}\label{eq:general_deficit_waiting_time}
		D(0)-D(T)&\geq\frac{\kappa}{4}\left(T-\frac{1-e^{-4T}}{4}\right),\\
		T&\leq\frac{4D(0)}{\kappa}+\frac14.
	\end{align}
	For $0\leq s\leq T$, monotonicity gives $D(s)\leq D(0)$. Together with \eqref{eq:general_deficit_tail_decay} and \eqref{eq:general_deficit_waiting_time}, this proves \eqref{eq:general_deficit_decay}, with $C$ depending only on an upper bound of $D(0)$.
\end{proof}

\begin{lemma}\label{lem:general_entropy_energy_decay}
	There is a constant $C>0$, depending only on an upper bound for $D(\varphi)$, such that
	\begin{align}\label{eq:general_entropy_energy_decay}
		E(s)\leq Ce^{-\frac12s},\quad \Ent(s)\leq Ce^{-\frac12s},\quad s\geq0.
	\end{align}
\end{lemma}

\begin{proof}
	By \eqref{eq:general_deficit_kinetic_identity} and Lemma \ref{lem:general_deficit_decay}, there is a constant $C_1>0$, depending only on an upper bound for $D(\varphi)$, such that
	$$\oint_{\bS^2}|\nabla f_s(s,x)|^2\,d\mu=I(s)\leq2D(s)\leq2C_1e^{-\frac12s}.$$
	By \eqref{eq:general_potential} and the uniform convergence $u\rightarrow1$, we have $f(s,\cdot)\rightarrow0$ in $H^1(\bS^2)$. Thus, integrating $\nabla f_s$ from $s$ to infinity and applying the Cauchy--Schwarz inequality, we obtain
	\begin{align}
		E(s)&=\oint_{\bS^2}|\int_s^\infty-\nabla f_s(\tau,x)\,d\tau|^2\,d\mu\\
		&\leq\oint_{\bS^2}(\int_s^\infty|\nabla f_s(\tau,x)|\,d\tau)^2\,d\mu\\
		&\leq\oint_{\bS^2}(\int_s^\infty|\nabla f_s(\tau,x)|^2e^{\frac14\tau}\,d\tau)(\int_s^\infty e^{-\frac14\tau}\,d\tau)\,d\mu\\
		&=4e^{-\frac14s}\int_s^\infty\oint_{\bS^2}e^{\frac14\tau}|\nabla f_s(\tau,x)|^2\,d\mu\,d\tau\\
		&\leq4e^{-\frac14s}\int_s^\infty2C_1e^{-\frac14\tau}\,d\tau=32C_1e^{-\frac12s}.
	\end{align}
	Hence, by \eqref{eq:general_deficit},
	$$\Ent(s)=D(s)+E(s)\leq33C_1e^{-\frac12s}.$$
	This proves \eqref{eq:general_entropy_energy_decay} with $C=33C_1$. The estimates hold at $s=0$ by continuity.
\end{proof}

\begin{remark}\label{rmk:general_normalized_entropy}
	At $s=0$, Lemma \ref{lem:general_entropy_energy_decay} bounds $\Ent(\varphi)$ in terms of $D(\varphi)$ only when $\varphi$ is the boundary value of a solution with $u\rightarrow1$. It does not give an inequality for arbitrary positive functions of mean one. For example, a nonconstant conformal Jacobian $\varphi=J_\phi$ has $D(\varphi)=0$ by Lemma \ref{lem:general_deficit_invariance}, but its stationary solution $u(s,x)=J_\phi(x)$ has limit $J_\phi$, not $1$.
\end{remark}

Next, we refine the $L^2$ bound argument using the entropy bound, so that only the ordinary Gagliardo--Nirenberg inequality on $\bS^2$ is needed.

\begin{lemma}\label{lem:general_second_moment}
	There is a constant $C>0$, depending only on upper bounds for $D(\varphi)$ and $\|\varphi\|_{L^2(\bS^2)}$, such that $F_2(s)\leq C$ for every $s\geq0$.
\end{lemma}

\begin{proof}
	As in Lemma \ref{lem:second_moment_ode},
	\begin{align}\label{eq:general_second_moment_identity}
		F_2''-3F_2'+4F_2&=4+2\oint_{\bS^2}(u_s^2+\frac{|\nabla u|^2}{u})d\mu\\
		&\geq4+8\oint_{\bS^2}|\nabla\sqrt{u}|^2d\mu.
	\end{align}
	Let $C_2$ be the constant in Lemma \ref{lem:general_entropy_energy_decay}. Since $r\log r\geq-e^{-1}$ for $r>0$, for every $L>1$ we have
	\begin{align}\label{eq:general_entropy_tail}
		\oint_{\{u>L\}}u\,d\mu\leq\frac{1}{\log L}\oint_{\{u>L\}}u\log u\,d\mu\leq\frac{C_2+e^{-1}}{\log L}.
	\end{align}
	Instead of applying the improved inequality to $\sqrt{u}$ as in \eqref{eq:second_moment_gns_step}, we apply the ordinary Gagliardo--Nirenberg inequality on $\bS^2$ to $w=(\sqrt{u}-\sqrt{L})_+$. Using \eqref{eq:general_entropy_tail}, we obtain
	\begin{align}\label{eq:general_truncated_gns}
		\oint_{\bS^2}w^4d\mu
		&\leq C\left(\oint_{\bS^2}|\nabla w|^2d\mu\right)\left(\oint_{\bS^2}w^2d\mu\right)+C\left(\oint_{\bS^2}w^2d\mu\right)^2\\
		&\leq\frac{C(C_2+e^{-1})}{\log L}\oint_{\bS^2}|\nabla\sqrt{u}|^2d\mu+C\left(\frac{C_2+e^{-1}}{\log L}\right)^2.
	\end{align}
	Choose $L=L(C_2)$ sufficiently large that $8C(C_2+e^{-1})/\log L\leq1$. Splitting into $\{u\leq L\}$ and $\{u>L\}$ gives
	\begin{align}
		F_2(s)
		&\leq L^2\mu(\{u\leq L\})+\oint_{\{u>L\}}(w+\sqrt{L})^4d\mu
		\\
		&\leq  L^2+8\left(\oint_{\bS^2}w^4d\mu+L^2\right)\leq\oint_{\bS^2}|\nabla\sqrt{u}|^2d\mu+C(C_2).
	\end{align}
	Substituting this into \eqref{eq:general_second_moment_identity} gives
	\begin{align}\label{eq:general_second_moment_ode}
		F_2''-3F_2'-4F_2\geq-C(C_2).
	\end{align}
	The function $F_2$ is bounded by \eqref{eq:boundary_asymptotics}. It follows that
	$$F_2(s)\leq\max\{F_2(0),C(C_2)/4\}.$$
\end{proof}

\begin{lemma}\label{lem:general_upper_bound}
	There is a constant $C>0$, depending only on upper bounds for $D(\varphi)$ and $M_0$, such that
	\begin{align}\label{eq:general_upper_bound}
		\|u\|_{L^\infty([0,\infty)\times\bS^2)}\leq C.
	\end{align}
\end{lemma}

\begin{proof}
	Apply the iteration \eqref{eq: M_q iteration} from the proof of Proposition \ref{prop:upper_bound}, starting from Lemma \ref{lem:general_second_moment}.
\end{proof}

\begin{lemma}\label{lem:general_tail_lower}
	Let $C$ be the constant in \eqref{eq:general_entropy_energy_decay} and let $\epsilon_*$ be given by Proposition \ref{prop:local_de_giorgi}. Set
	\begin{align}\label{eq:general_tail_threshold}
		S=\max\left\{2,\,2+2\log\left(\frac{32C}{\epsilon_*}\right)\right\}.
	\end{align}
	Then
	\begin{align}\label{eq:general_tail_lower}
		u\geq\frac14\quad\mathrm{on}\quad[S,\infty)\times\bS^2.
	\end{align}
	In particular, $S$ depends only on an upper bound for $D(\varphi)$.
\end{lemma}

\begin{proof}
	Using the sublevel-set estimate in the proof of Proposition \ref{prop:tail_lower} and Lemma \ref{lem:general_entropy_energy_decay}, we obtain, for every $R\geq S$,
	\begin{align}
		|\{u<\frac12\}\cap Q_2(R)|
		&\leq8\int_{R-2}^{R+2}\Ent(s)\,ds\\
		&<32Ce^{-\frac12(R-2)}\leq\epsilon_*.
	\end{align}
	Proposition \ref{prop:local_de_giorgi} gives $u\geq1/4$ on $Q_1(R)$ for every $R\geq S$, proving \eqref{eq:general_tail_lower}.
\end{proof}

Combining these estimates gives the following conclusion.

\begin{theorem}\label{thm:general_c0_decay}
	For $0<m\leq1$ and $0<\alpha<1$, there is a constant $C>1$, depending only on $m$ and $\alpha$, such that every positive solution of \eqref{eq:s_pde} and \eqref{eq:boundary_asymptotics} with
	$$m\leq\varphi\leq m^{-1}$$
	satisfies $C^{-1}\leq u\leq C$ on $[0,\infty)\times\bS^2$ and, with $\delta=1/22$,
	\begin{align}\label{eq:general_derivative_decay}
		\sup_{s\geq1}e^{\delta s}\|u-1\|_{C^{2,\alpha}([s,s+1]\times\bS^2)}\leq C.
	\end{align}
\end{theorem}

\begin{proof}
	By \eqref{eq:mass_one}, $0\leq D(\varphi)\leq\Ent(\varphi)\leq\log m^{-1}$. The $C^0$ bounds follow from Lemmas \ref{lem:general_upper_bound} and \ref{lem:general_tail_lower}, together with \eqref{eq:comparison_lower} and \eqref{eq:boundary_lower}, as in the proof of Theorem \ref{thm:c0_estimate}.
	
	As in the proof of \eqref{eq:moment_decay}, Lemma \ref{lem:general_entropy_energy_decay} and the upper bound give
	$$\oint_{\bS^2}(u-1)^2\,d\mu\leq2C\Ent(s)\leq Ce^{-\frac12s}.$$
	The interior estimates and interpolation in the proof of Corollary \ref{cor:derivative_decay} now give \eqref{eq:general_derivative_decay} with $\delta=1/22$.
\end{proof}

\subsection{Existence}\label{subsec:general_existence}

With $\mathfrak{j}_a$ as in \eqref{eq:general_stationary_family}, set
$$\mathcal{J}=\{\mathfrak{j}_a: a\in\bR^3\},$$
which is a non-compact 3-manifold diffeomorphic to $\bR^3$ via the map
$$\Psi:\bR^3\rightarrow\mathcal{J},\quad\Psi(a)=\mathfrak{j}_a.$$
For $\mathfrak{j}\in\mathcal{J}$, its tangent space is the kernel of the linearization of \eqref{eq:general_stationary_equation}:
$$T_{\mathfrak{j}}\mathcal{J}=\{b\in C^{2,\alpha}(\bS^2)|\,\Delta(\mathfrak{j}^{-1}b)+2b=0\}.$$
In particular, $T_{\mathfrak{j}_0}\mathcal{J}=T_1\mathcal{J}=\text{span}_{\bR}\{x_1,x_2,x_3\}$.

In this subsection, we prove the following.

\begin{theorem}\label{thm:general_existence_decay}
	For any positive $\varphi\in C^{2,\alpha}(\bS^2)$, $0<\alpha<1$, satisfying $\oint_{\bS^2}\varphi\,d\mu=1$, there exists a unique bounded positive classical solution of \eqref{eq:s_pde} with $u(0,\cdot)=\varphi$. Moreover, there exists a unique $a\in\bR^3$ such that
	$$\lim\limits_{s\rightarrow\infty}\sup_{x\in\bS^2}|u(s,x)-\mathfrak{j}_a(x)|=0.$$
	If $M^{-1}\leq\varphi\leq M$ and $\|\varphi\|_{C^{2,\alpha}(\bS^2)}\leq M$, then there is a constant $C>1$, depending only on $M$ and $\alpha$, such that, with $\delta=1/22$,
	\begin{align}\label{eq:general_existence_bounds}
		C^{-1}\leq u\leq C,\quad
		\sup_{r\geq0}e^{\delta r}\|u-\mathfrak{j}_a\|_{C^{2,\alpha}([r,r+1]\times\bS^2)}\leq C.
	\end{align}
\end{theorem}

\begin{proposition}\label{prop:general_limit_parameter_bound}
	Suppose $u>0$ is a classical solution of \eqref{eq:s_pde} with $u(0,\cdot)=\varphi$ and $u\rightarrow \mathfrak{j}_a$ uniformly. Then $|a|\leq C$, where $C$ depends only on an upper bound for $\Ent(\varphi)$.
\end{proposition}

\begin{proof}
	Choose a conformal transformation $\phi$ such that $(\mathfrak{j}_a)_\phi=1$. The transformed solution has boundary value $\varphi_\phi$ and limit $1$. By Lemma \ref{lem:general_deficit_invariance},
	$$D(\varphi_\phi)=D(\varphi)\leq\Ent(\varphi).$$
	Lemma \ref{lem:general_entropy_energy_decay}, evaluated at $s=0$, therefore gives $E(\varphi_\phi)\leq C$. Conformal invariance of the Dirichlet integral gives
	$$\|\varphi-\mathfrak{j}_a\|_{H^{-1}(\bS^2)}^2=\|\varphi_\phi-1\|_{H^{-1}(\bS^2)}^2=E(\varphi_\phi)\leq C.$$
	Hence,
	$$\Ent(\mathfrak{j}_a)=E(\mathfrak{j}_a)\leq2E(\varphi)+2\|\varphi-\mathfrak{j}_a\|_{H^{-1}(\bS^2)}^2\leq2\Ent(\varphi)+2C.$$
	A direct integration gives, for $|a|>0$,
	$$\Ent(\mathfrak{j}_a)=\frac{2\sqrt{1+|a|^2}}{|a|}\log(\sqrt{1+|a|^2}+|a|)-2.$$
	This tends to infinity as $|a|\rightarrow\infty$, proving the bound.
\end{proof}

For $0<\alpha<1$ and $\delta>0$, define $X_\delta$ and $Y_\delta$ as in \eqref{eq:x_delta}--\eqref{eq:y_delta}, but without the evenness condition:
\begin{align}\label{eq:general_x_delta}
	X_\delta=\{h\in C_\delta^{2,\alpha}|\,h(0,\cdot)=0,\,\oint_{\bS^2}h(s,x)\,d\mu=0,\,\forall s\geq0\},
\end{align}
\begin{align}\label{eq:general_y_delta}
	Y_\delta=\{f\in C_\delta^{0,\alpha}|\,\oint_{\bS^2}f(s,x)\,d\mu=0,\,\forall s\geq0\},
\end{align}
equipped with the induced norms
$$\|h\|_{X_\delta}=\|h\|_{C_\delta^{2,\alpha}},\quad \|f\|_{Y_\delta}=\|f\|_{C_\delta^{0,\alpha}},$$
where the weighted norms are defined in \eqref{eq:weighted_norm}. These are closed subspaces of the corresponding weighted H\"older spaces and hence are Banach spaces.

Without evenness, the degree-one mode ODE $v''-3v'=f$ need not admit a solution with $v(0)=0$ and $|v(s)|+|v'(s)|=O(e^{-\delta s})$, as required by $X_\delta$, even though $f(s)=O(e^{-\delta s})$. Indeed, these conditions require $\int_0^\infty(1-e^{-3s})f(s)\,ds=0$, which need not hold. Thus $L_1:X_\delta\rightarrow Y_\delta$ is not surjective.

For $\mathfrak{j}\in\mathcal{J}$, enlarge the domain to
\begin{align}\label{eq:general_extended_x_delta}
	X_\delta(\mathfrak{j})=\{(1-e^{-s})b+h|\,b\in T_{\mathfrak{j}}\mathcal{J},\,h\in X_\delta\}.
\end{align}
For $v=(1-e^{-s})b+h\in X_\delta(\mathfrak{j})$, set
\begin{align}\label{eq:general_extended_x_norm}
	\|v\|_{X_\delta(\mathfrak{j})}=\|b\|_{C^{2,\alpha}(\bS^2)}+\|h\|_{X_\delta}.
\end{align}
The decomposition is unique, since $v(s,\cdot)\rightarrow b$, so this is a norm.

\begin{proposition}\label{prop:general_linear_isomorphism}
	Let $0<\delta<1$, $0<\alpha<1$, and $\mathfrak{j}\in\mathcal{J}$. If $u>0$ and $u-\mathfrak{j}\in C_\delta^{2,\alpha}$, then
	\begin{align}\label{eq:general_linear_isomorphism}
		L_u:X_\delta(\mathfrak{j})\rightarrow Y_\delta
	\end{align}
	is an isomorphism of Banach spaces.
\end{proposition}

\begin{proof}
	For $v=(1-e^{-s})b+h\in X_\delta(\mathfrak{j})$, the identity $L_{\mathfrak{j}}b=0$ gives
	\begin{align}\label{eq:general_linear_decomposition}
		L_uv=L_uh-4e^{-s}b+(1-e^{-s})\Delta((u^{-1}-\mathfrak{j}^{-1})b).
	\end{align}
	Since $u^{-1}-\mathfrak{j}^{-1}\in C_\delta^{2,\alpha}$ and $\delta<1$, the right-hand side belongs to $C_\delta^{0,\alpha}$ with norm at most $C\|v\|_{X_\delta(\mathfrak{j})}$. Integrating $L_{\mathfrak{j}}b=0$ over $\bS^2$ gives $\oint_{\bS^2}b\,d\mu=0$, so $L_uv$ also has zero spherical mean. Thus $L_u$ maps $X_\delta(\mathfrak{j})$ continuously into $Y_\delta$. Injectivity follows from the proof of Lemma \ref{lem:linear_isomorphism}.

	For $\mathfrak{j}=1$, we solve $L_1v=f$ by spherical harmonics. The modes of degree at least two are estimated exactly as in Lemma \ref{lem:linear_isomorphism}. For each degree-one mode, the bounded solution with $v_{1m}(0)=0$ is determined by
	$$v_{1m}'(s)=-\int_s^\infty e^{-3(r-s)}f_{1m}(r)\,dr.$$
	It has a finite limit $b_{1m}$, and $v_{1m}-b_{1m}$ decays at rate $e^{-\delta s}$. Since $\delta<1$, these modes belong to $X_\delta(1)$. Together with the Schauder estimates used in \eqref{eq:model_inverse_estimate}, this proves surjectivity and
	$$\|v\|_{X_\delta(1)}\leq C\|L_1v\|_{Y_\delta}.$$
	A conformal change sending $\mathfrak{j}$ to $1$ identifies $X_\delta(\mathfrak{j})$ with $X_\delta(1)$ and gives the same conclusion for $L_{\mathfrak{j}}$.

	For general $u$, set $u_\theta=\mathfrak{j}+\theta(u-\mathfrak{j})$, $0\leq\theta\leq1$. Apply the cutoff argument for \eqref{eq:linear_compact_remainder} to $h=v-(1-e^{-s})b$, using the model estimate for $L_{\mathfrak{j}}$. Applying \eqref{eq:general_linear_decomposition} with $u_\theta$ in place of $u$ and $h=0$, we obtain
	$$\|L_{u_\theta}((1-e^{-s})b)\|_{Y_\delta}\leq C\|b\|_{C^{2,\alpha}(\bS^2)},$$
	uniformly in $\theta$. Since $h=v-(1-e^{-s})b$, the resulting estimate for $h$ gives
	\begin{align}\label{eq:general_linear_compact_remainder}
		\|v\|_{X_\delta(\mathfrak{j})}
		\leq C\left(\|L_{u_\theta}v\|_{Y_\delta}
		+\|v\|_{C^0([0,R]\times\bS^2)}
		+\|b\|_{C^{2,\alpha}(\bS^2)}\right),
	\end{align}
	uniformly in $\theta$. We claim that
	\begin{align}\label{eq:general_uniform_linear_estimate}
		\|v\|_{X_\delta(\mathfrak{j})}\leq C\|L_{u_\theta}v\|_{Y_\delta}
	\end{align}
	for every $v\in X_\delta(\mathfrak{j})$, with $C$ independent of $\theta\in[0,1]$. If \eqref{eq:general_uniform_linear_estimate} failed, there would be $\theta_i\rightarrow\theta_\infty$ and $v_i=(1-e^{-s})b_i+h_i$ such that
	$$\|v_i\|_{X_\delta(\mathfrak{j})}=1,\quad \|L_{u_{\theta_i}}v_i\|_{Y_\delta}\rightarrow0.$$
	After passing to a subsequence, finite dimensionality gives $b_i\rightarrow b_\infty$ in $C^{2,\alpha}(\bS^2)$, and the compactness argument in Lemma \ref{lem:weighted_local_compactness} gives $h_i\rightarrow h_\infty\in X_\delta$ locally in $C^{2,\beta}$, $0<\beta<\alpha$. Then $v_\infty=(1-e^{-s})b_\infty+h_\infty$ satisfies $L_{u_{\theta_\infty}}v_\infty=0$, so injectivity gives $v_\infty=0$. Since $h_\infty$ tends to zero at infinity, we also have $b_\infty=0$. Thus both $\|v_i\|_{C^0([0,R]\times\bS^2)}$ and $\|b_i\|_{C^{2,\alpha}(\bS^2)}$ tend to zero. This contradicts \eqref{eq:general_linear_compact_remainder} and $\|v_i\|_{X_\delta(\mathfrak{j})}=1$.

	The estimate \eqref{eq:general_uniform_linear_estimate} and the operator-norm continuity of $L_{u_\theta}$ give the same method of continuity as in Lemma \ref{lem:linear_isomorphism}, starting from the isomorphism $L_{\mathfrak{j}}$.
\end{proof}

We now apply the continuity method as in Section \ref{subsec:continuity_method}.
Fix $\delta=1/22$ and a positive $\varphi\in C^{2,\alpha}(\bS^2)$ with $\oint_{\bS^2}\varphi\,d\mu=1$, and set
$$\varphi_\tau=1-\tau+\tau\varphi,\quad \tau\in[0,1].$$
The bounds \eqref{eq:path_bounds} hold without evenness. In particular, $\Ent(\varphi_\tau)$ and $\|\varphi_\tau\|_{C^{2,\alpha}(\bS^2)}$ are uniformly bounded, and $\varphi_\tau$ has a uniform positive lower bound.

For every positive solution with boundary value $\varphi_\tau$ and limit $\mathfrak{j}_a$, Proposition \ref{prop:general_limit_parameter_bound} bounds $|a|$ uniformly. We may therefore normalize $\mathfrak{j}_a$ to $1$ by conformal transformations in a compact family. The transformed boundary values have uniform positive lower bounds and $C^{2,\alpha}$ bounds. Theorem \ref{thm:general_c0_decay}, together with the boundary estimates in the proof of Corollary \ref{cor:derivative_decay}, gives a constant $C>1$, independent of $\tau$ and the solution, such that, after transforming back,
\begin{align}\label{eq:general_path_weighted_bounds}
	|a|\leq C,\quad C^{-1}\leq u\leq C,\quad
	\|u-\mathfrak{j}_a\|_{C_\delta^{2,\alpha}}\leq C.
\end{align}
For $\tau\in[0,1]$ and $a\in\bR^3$, define the extension of $\varphi_\tau$ by
\begin{align}\label{eq:general_boundary_extension}
	E_{\tau,a}=\mathfrak{j}_a+e^{-s}(\varphi_\tau-\mathfrak{j}_a)\in \mathfrak{j}_a+C_\delta^{2,\alpha}.
\end{align}
It has boundary value $\varphi_\tau$ and spherical mean one. Thus $u-E_{\tau,a}\in X_\delta$, and \eqref{eq:general_path_weighted_bounds} gives
\begin{align}\label{eq:general_path_h_bound}
	\|u-E_{\tau,a}\|_{X_\delta}
	\leq\|u-\mathfrak{j}_a\|_{C_\delta^{2,\alpha}}+\|E_{\tau,a}-\mathfrak{j}_a\|_{C_\delta^{2,\alpha}}\leq C.
\end{align}
Set
\begin{align}\label{eq:general_continuity_set}
	\mathcal I=\{\tau\in[0,1]|\,\mathcal{P}(E_{\tau,a}+h)=0,\ E_{\tau,a}+h>0
	\text{ for some }a\in\bR^3,\ h\in X_\delta\}.
\end{align}
We show that $\mathcal I=[0,1]$.

\begin{itemize}
	\item $\mathcal I$ is non-empty: $0\in\mathcal I$, since $a=0$ and $h=0$ give $u\equiv1$.

	\item $\mathcal I$ is closed: let $\tau_i\in\mathcal I$ with $\tau_i\rightarrow\tau_\infty$, and write the corresponding solutions as $u_i=E_{\tau_i,a_i}+h_i$. By \eqref{eq:general_path_weighted_bounds}--\eqref{eq:general_path_h_bound}, after passing to a subsequence, $a_i\rightarrow a_\infty$. The proof of Lemma \ref{lem:weighted_local_compactness} gives $h_i\rightarrow h_\infty$ locally in $C^{2,\beta}$, $0<\beta<\alpha$, with $h_\infty\in X_\delta$. Thus $u_i\rightarrow u_\infty=E_{\tau_\infty,a_\infty}+h_\infty$ locally. The lower bound in \eqref{eq:general_path_weighted_bounds} gives $u_\infty>0$, and passing the equation to the limit yields $\mathcal{P}(u_\infty)=0$. Hence $\tau_\infty\in\mathcal I$.

	\item $\mathcal I$ is open: We use the implicit function theorem. Assume $\tau_0\in\mathcal I$. For $(\tau,a,h)\in\bR\times\bR^3\times X_\delta$ with $E_{\tau,a}+h>0$, set
	\begin{align}\label{eq:general_implicit_function_map}
		\mathcal{F}(\tau,a,h)=\mathcal{P}(E_{\tau,a}+h)\in Y_\delta.
	\end{align}
	Then $\mathcal F$ is smooth on its domain.
	
	At a solution $u=E_{\tau,a}+h$, for $(\eta,v)\in\bR^3\times X_\delta$, differentiation gives
	\begin{align}\label{eq:general_implicit_derivative}
		D_{(a,h)}\mathcal{F}(\tau,a,h)[\eta,v]
		=L_u\bigl((1-e^{-s})D\Psi(a)[\eta]+v\bigr).
	\end{align}
	The map $\eta\mapsto D\Psi(a)[\eta]$ identifies $\bR^3$ with $T_{\mathfrak{j}_a}\mathcal{J}$. Thus $(\eta,v)\mapsto(1-e^{-s})D\Psi(a)[\eta]+v$ is an isomorphism onto $X_\delta(\mathfrak{j}_a)$. Proposition \ref{prop:general_linear_isomorphism} shows that $D_{(a,h)}\mathcal{F}$ is an isomorphism. The implicit function theorem therefore gives solutions for all $\tau'\in[0,1]$ sufficiently close to $\tau$, proving openness.
\end{itemize}

Hence $\mathcal I=[0,1]$, and $\tau=1$ gives the desired solution. The comparison argument at the beginning of Section \ref{sec:general_case} gives uniqueness among bounded positive solutions, since boundedness gives \eqref{eq:intro_degenerate_end} under \eqref{eq:ball_transform}. The parameter $a$ is unique because $\Psi$ is one-to-one. Finally, \eqref{eq:general_path_weighted_bounds} gives \eqref{eq:general_existence_bounds} with $\delta=1/22$ and $C$ depending only on $M$ and $\alpha$. This proves Theorem \ref{thm:general_existence_decay}, and hence Theorem \ref{thm:general_existence_uniqueness}.

Interior Schauder estimates applied to $u-\mathfrak{j}_a$ give decay of every derivative with the same rate $\delta=1/22$ as in \eqref{eq:general_existence_bounds}. Together with the boundary estimate in \eqref{eq:general_existence_bounds} and \eqref{eq:ball_transform}, this gives Theorem \ref{thm:intro_general_case}.

\section{The line bundle case}\label{sec:line_bundle_case}

In this section, we consider the $SU(\infty)$-Toda equation
\begin{align}\label{eq:line_bundle_toda}
	(e^w)_{\xi\xi}+\Delta_{\bS^2}w=2
\end{align}
on $[0,\frac12)\times\bS^2$, with boundary conditions
\begin{align}\label{eq:line_bundle_toda_boundary}
	w(0,\cdot)=\psi,\quad w(\xi,x)=\log(\frac12-\xi)+O(1)\quad\text{as }\xi\rightarrow\frac12,
\end{align}
where $\psi\in C^{2,\alpha}(\bS^2)$, $0<\alpha<1$, and the $O(1)$ term is uniform in $x$.

As in \cite[Section 4.2]{LiLiu2025}, make the transformation
\begin{align}\label{eq:line_bundle_transform}
	e^w=(\frac12-\xi)u,\quad \xi=\frac12-\frac14|Z|^2,\quad V(Z,x)=u(\frac12-\frac14|Z|^2,x).
\end{align}
For $Z\in\bR^4$ with $0<|Z|<\sqrt{2}$, a direct computation gives
$$\Delta_{\bR^4}V=(\frac12-\xi)u_{\xi\xi}-2u_\xi=\bigl((\frac12-\xi)u\bigr)_{\xi\xi}.$$
Thus \eqref{eq:line_bundle_toda} transforms to
\begin{align}\label{eq:line_bundle_pde}
	\Delta_{\bR^4}V+\Delta_{\bS^2}\log V=2
\end{align}
on $\Omega=B_{\sqrt{2}}\times\bS^2$ away from $Z=0$, where $B_{\sqrt{2}}$ is the ball centered at $0$ in $\bR^4$. The boundary condition at $\xi=0$ becomes
\begin{align}\label{eq:line_bundle_boundary}
	V|_{\partial\Omega}=\varphi(x).
\end{align}
Here $\varphi=2e^\psi$.
Throughout this section, we consider positive solutions $V\in C^{2,\alpha}(\overline{\Omega})$ that solve \eqref{eq:line_bundle_pde} on $\Omega$.

\begin{remark}
	The radial change of variables is the same as in Li--Liu \cite[Section 4.2]{LiLiu2025}, but here we lift $u=e^w/(\frac12-\xi)$ instead of $w-\bar{w}$, where
	$e^{\bar{w}(\xi)}=\oint_{\bS^2}e^{w(\xi,x)}\,d\mu.$
	This gives the simpler equation \eqref{eq:line_bundle_pde}, to which the comparison argument in the proof of Lemma \ref{lem:line_bundle_comparison_bounds} applies directly.
\end{remark}

For a function $g=g(Z,x)$, set $\bar{g}(Z)=\oint_{\bS^2}g(Z,x)\,d\mu$. Integrating \eqref{eq:line_bundle_pde} over $\bS^2$, we obtain
$$\Delta_{\bR^4}\overline{V}=2,\quad \overline{V}|_{\partial B_{\sqrt{2}}}=\oint_{\bS^2}\varphi\,d\mu.$$
Since $\Delta_{\bR^4}|Z|^2=8$, uniqueness for this Dirichlet problem gives
\begin{align}\label{eq:line_bundle_mean}
	\overline{V}(Z)=\oint_{\bS^2}\varphi\,d\mu-\frac12+\frac14|Z|^2.
\end{align}
Since $V>0$ at $Z=0$, we have the necessary condition
\begin{align}\label{eq:integral constraint line bundle case}
	\oint_{\bS^2}\varphi\,d\mu>\frac12.
\end{align}

\subsection{A priori $C^0$ estimates}

Set
\begin{align}\label{eq:line_bundle_boundary_constants}
	m_1=\inf_{\bS^2}\varphi,\quad M_1=\sup_{\bS^2}\varphi.
\end{align}

\begin{lemma}\label{lem:line_bundle_comparison_bounds}
	Every positive solution of \eqref{eq:line_bundle_pde}--\eqref{eq:line_bundle_boundary} satisfies
	\begin{align}\label{eq:line_bundle_comparison_bounds}
		\max\{0,m_1-\frac12\}\leq V\leq M_1\quad\text{on }\overline{\Omega}.
	\end{align}
\end{lemma}

\begin{proof}
	The positive-part argument from Lemma \ref{lem:comparison} gives comparison for \eqref{eq:line_bundle_pde}: for two positive solutions $V,W$, the function $\oint_{\bS^2}(V-W)_+\,d\mu$ is subharmonic on $B_{\sqrt{2}}$. Hence $V\leq W$ on $\partial\Omega$ implies $V\leq W$ in $\overline{\Omega}$.
	
	Consider $V_M=M_1-\frac12+\frac14|Z|^2$. By \eqref{eq:integral constraint line bundle case}, $M_1>\frac12$, so $V_M>0$. It solves \eqref{eq:line_bundle_pde}, and $V\leq V_M$ on $\partial\Omega$. Thus $V\leq V_M\leq M_1$ in $\overline{\Omega}$.
	If $m_1>\frac12$, consider $V_m=m_1-\frac12+\frac14|Z|^2>0$. Since $V\geq V_m$ on $\partial\Omega$, comparison gives $V\geq V_m\geq m_1-\frac12$ in $\overline{\Omega}$. If $m_1\leq\frac12$, the lower estimate follows from $V>0$.
\end{proof}

In particular, if $m_1>\frac12$, we have an a priori positive lower bound for $V$. The next lemma gives a lower bound under the natural constraint \eqref{eq:integral constraint line bundle case}, without this stronger assumption on $m_1$.

\begin{lemma}\label{lem:line_bundle_lower_bound}
	For each $n_1>1$, there is a constant $c=c(n_1)>0$ such that, if
	\begin{align}\label{eq:line_bundle_quantitative_data}
		n_1^{-1}\leq\oint_{\bS^2}\varphi\,d\mu-\frac12\leq n_1,\quad n_1^{-1}\leq\varphi\leq n_1,
	\end{align}
	then every positive solution of \eqref{eq:line_bundle_pde}--\eqref{eq:line_bundle_boundary} satisfies $V\geq c$ on $\overline{\Omega}$.
\end{lemma}

\begin{proof}
	Let $f(Z,x)$ be the slicewise mean-zero solution of
	\begin{align}\label{eq:line_bundle_potential}
		-\Delta_{\bS^2}f=V-\overline{V},\quad \bar{f}=0.
	\end{align}
	By Green's function representation and Lemma \ref{lem:line_bundle_comparison_bounds},
	$$\|f(Z,\cdot)\|_{L^\infty(\bS^2)}\leq C_0\|V(Z,\cdot)-\overline{V}(Z)\|_{L^\infty(\bS^2)}\leq2C_0n_1,$$
	where $C_0$ is universal. Thus
	\begin{align}\label{eq:line_bundle_potential_bound}
		\|f\|_{L^\infty(\Omega)}\leq2C_0n_1.
	\end{align}
	By \eqref{eq:line_bundle_pde} and \eqref{eq:line_bundle_mean},
	$$\Delta_{\bS^2}\Delta_{\bR^4}f=-\Delta_{\bR^4}V+\Delta_{\bR^4}\overline{V}=\Delta_{\bS^2}\log V.$$
	Since $\bar{f}=0$, it follows that
	\begin{align}\label{eq:line_bundle_potential_z}
		\Delta_{\bR^4}f=\log V-\overline{\log V}.
	\end{align}
	
	Consider the linear operator
	\begin{align}\label{eq:line_bundle_adjoint}
		\widetilde{\mathcal{L}}_V=\Delta_{\bR^4}+V^{-1}\Delta_{\bS^2},
	\end{align}
	which is uniformly elliptic and satisfies the maximum principle. Let $A>0$ be determined below, and suppose that $\log V-Af$ achieves its minimum at $p\in\overline{\Omega}$. If $p\in\Omega$, then at $p$, using \eqref{eq:line_bundle_potential} and \eqref{eq:line_bundle_potential_z}, we have
	\begin{align}
		0\leq\widetilde{\mathcal{L}}_V(\log V-Af)
		&=\Delta_{\bR^4}\log V+V^{-1}\Delta_{\bS^2}\log V-A\Delta_{\bR^4}f-AV^{-1}\Delta_{\bS^2}f\\
		&=\Delta_{\bR^4}\log V+V^{-1}(2-\Delta_{\bR^4}V)-A(\log V-\overline{\log V})+AV^{-1}(V-\overline{V})\\
		&=-|\nabla_{\bR^4}\log V|^2+2V^{-1}-A\log V+A\overline{\log V}+A-A\frac{\overline{V}}{V}\\
		&\leq2V^{-1}-A\log V+A\log(\overline{V})+A-A\frac{\oint_{\bS^2}\varphi\,d\mu-\frac12}{V}\\
		&\leq2V^{-1}-A\log V+A\log(\oint_{\bS^2}\varphi\,d\mu)+A-A\frac{\oint_{\bS^2}\varphi\,d\mu-\frac12}{V}.
	\end{align}
	Here we used Jensen's inequality and \eqref{eq:line_bundle_mean}. Choose $A$ so that
	$$A\left(\oint_{\bS^2}\varphi\,d\mu-\frac12\right)=3.$$
	Then $3n_1^{-1}\leq A\leq3n_1$, and
	$$0\leq-V(p)^{-1}-A\log V(p)+A\left(1+\log\left(\oint_{\bS^2}\varphi\,d\mu\right)\right).$$
	Multiplying by $V(p)$ and using $r|\log r|\rightarrow0$ as $r\rightarrow0$ shows that $V(p)\geq c_0(n_1)>0$.
	
	Suppose $V$ achieves its minimum at $q\in\overline{\Omega}$. We use \eqref{eq:line_bundle_potential_bound} to transfer the estimate at $p$ to $q$.
	\begin{itemize}
		\item If $q\in\partial\Omega$, then $V(q)\geq\inf_{\bS^2}\varphi\geq n_1^{-1}$.
		\item If $q\in\Omega$ and $p\in\Omega$, then
		$$\log V(q)\geq\log V(p)+A(f(q)-f(p))\geq\log c_0(n_1)-4AC_0n_1.$$
		Thus $V(q)\geq c_0(n_1)e^{-12C_0n_1^2}$.
		\item If $q\in\Omega$ and $p\in\partial\Omega$, then
		$$\log V(q)\geq\log V(p)+A(f(q)-f(p))\geq\log(n_1^{-1})-4AC_0n_1.$$
		Thus $V(q)\geq n_1^{-1}e^{-4AC_0n_1}\geq n_1^{-1}e^{-12C_0n_1^2}$.
	\end{itemize}
	All three bounds depend only on $n_1$, proving the lemma.
\end{proof}

\subsection{Existence}

With the a priori upper and lower bounds, we can establish existence by the continuity method as in Section \ref{subsec:continuity_method}. The argument is simpler here because $\Omega$ is bounded, so no estimates at infinity are needed.

Consider the nonlinear operator
\begin{align}\label{eq:line_bundle_nonlinear_operator}
	\mathcal{P}(V)=\Delta_{\bR^4}V+\Delta_{\bS^2}\log V-2,
\end{align}
whose linearization at $V$ is
\begin{align}\label{eq:line_bundle_linearization}
	\mathcal{L}_VW=\Delta_{\bR^4}W+\Delta_{\bS^2}(V^{-1}W).
\end{align}
Set
$$\mathcal{X}=\{h\in C^{2,\alpha}(\overline{\Omega})|\,h|_{\partial\Omega}=0\},\quad\mathcal{Y}=C^{0,\alpha}(\overline{\Omega}),$$
equipped with the induced norms.

\begin{lemma}\label{lem:line_bundle_linear_isomorphism}
	For every positive $V\in C^{2,\alpha}(\overline{\Omega})$, the operator $\mathcal{L}_V:\mathcal{X}\rightarrow\mathcal{Y}$ is an isomorphism of Banach spaces.
\end{lemma}

\begin{proof}
	The positive-part comparison argument above shows that
	$$\mathcal{L}_VW=0,\quad W|_{\partial\Omega}=0$$
	has only the trivial solution: $\oint_{\bS^2}W_+\,d\mu$ is subharmonic on $B_{\sqrt{2}}$ and vanishes on its boundary, and the same argument applies to $-W$.
	
	Moreover, the formal adjoint of $\mathcal{L}_V$ is $\widetilde{\mathcal{L}}_V$ from \eqref{eq:line_bundle_adjoint}. By the maximum principle,
	$$\widetilde{\mathcal{L}}_V\widetilde{W}=0,\quad\widetilde{W}|_{\partial\Omega}=0$$
	also has only the trivial solution. The Fredholm alternative for the uniformly elliptic Dirichlet problem therefore gives surjectivity. Together with injectivity and the Schauder estimate, this proves the lemma.
\end{proof}

We can now complete the continuity argument.

\begin{theorem}\label{thm:line_bundle_existence}
	For any positive $\varphi\in C^{2,\alpha}(\bS^2)$, $0<\alpha<1$, satisfying
	$$\oint_{\bS^2}\varphi\,d\mu>\frac12,$$
	there is a unique positive $V\in C^{2,\alpha}(\overline{\Omega})$ solving \eqref{eq:line_bundle_pde}--\eqref{eq:line_bundle_boundary}. Moreover, $V$ is $O(4)$--invariant in $Z$, that is, $V(Z_1,x)=V(Z_2,x)$ whenever $|Z_1|=|Z_2|$.
\end{theorem}

\begin{proof}
	Choose $n_1\geq2$ so that $\varphi$ satisfies \eqref{eq:line_bundle_quantitative_data}. Set $\varphi_\tau=1-\tau+\tau\varphi$ for $\tau\in[0,1]$. This path satisfies \eqref{eq:line_bundle_quantitative_data} with the same $n_1$, since $\varphi_\tau$ is uniformly positive and bounded, and
	$$\oint_{\bS^2}\varphi_\tau\,d\mu-\frac12=(1-\tau)\frac12+\tau\left(\oint_{\bS^2}\varphi\,d\mu-\frac12\right)\geq n_1^{-1}.$$
	By Lemmas \ref{lem:line_bundle_comparison_bounds} and \ref{lem:line_bundle_lower_bound}, every positive solution with boundary value $\varphi_\tau$ satisfies $c\leq V\leq C$, independently of $\tau$ and $V$. Written in divergence form, the equation is
	$$\Delta_{\bR^4}V+\mathrm{div}_{\bS^2}(V^{-1}\nabla_{\bS^2}V)=2.$$
	Interior and boundary De Giorgi--Nash--Moser estimates, followed by Schauder estimates, give
	\begin{align}\label{eq:line_bundle_schauder_bound}
		\|V\|_{C^{2,\alpha}(\overline{\Omega})}\leq C,
	\end{align}
	where $C$ depends only on $n_1,\alpha$, and $\|\varphi\|_{C^{2,\alpha}(\bS^2)}$.
	
	Let $\mathcal{I}$ be the set of $\tau\in[0,1]$ for which the problem with boundary value $\varphi_\tau$ has a positive solution in $C^{2,\alpha}(\overline{\Omega})$.
	\begin{itemize}
		\item $\mathcal{I}$ is non-empty: at $\tau=0$, the function $V=\frac12+\frac14|Z|^2$ is a positive solution with boundary value $1$.
		\item $\mathcal{I}$ is closed: if $\tau_i\in\mathcal{I}$ and $\tau_i\rightarrow\tau_\infty$, then \eqref{eq:line_bundle_schauder_bound} gives a subsequence of corresponding solutions converging in $C^{2,\beta}(\overline{\Omega})$, $0<\beta<\alpha$, to $V_\infty\in C^{2,\alpha}(\overline{\Omega})$. The uniform lower bound gives $V_\infty>0$. Passing to the limit in the equation and boundary condition gives $\tau_\infty\in\mathcal{I}$.
		\item $\mathcal{I}$ is open: For $(\tau,h)\in\bR\times\mathcal{X}$ with $\varphi_\tau+h>0$, define
		$$
		\mathcal{F}(\tau,h)=\mathcal{P}(\varphi_\tau+h)\in\mathcal{Y}.
		$$
		Then $\mathcal{F}$ is smooth on its domain, and
		$$
		D_h\mathcal{F}(\tau,h)=\mathcal{L}_{\varphi_\tau+h}
		$$
		is an isomorphism by Lemma \ref{lem:line_bundle_linear_isomorphism}. The implicit function theorem gives openness.
	\end{itemize}
	Thus $\mathcal{I}=[0,1]$, giving existence at $\tau=1$. Uniqueness follows from the comparison argument in the proof of Lemma \ref{lem:line_bundle_comparison_bounds}. For each $R\in O(4)$, the function $V(RZ,x)$ solves the same equation with the same boundary value. Uniqueness therefore gives $O(4)$ invariance.
\end{proof}

Interior Schauder estimates give bounds for all derivatives of $V$ away from $\partial\Omega$. Since $V$ is smooth and radial in $Z$, it is smooth in $|Z|^2$ at $Z=0$, with $j$ derivatives in $\xi$ controlled by $2j$ derivatives in $Z$. Together with \eqref{eq:line_bundle_schauder_bound}, \eqref{eq:line_bundle_transform}, and Corollary \ref{cor:uniqueness}, these bounds give Theorem \ref{thm:intro_line_case}.

We can now complete the proof of Theorem \ref{thm:intro_asymptotics}. If
$
\oint_{\bS^2}e^\psi\,d\mu=\frac14,
$
Theorem \ref{thm:intro_general_case} gives a solution with the asymptotic behavior in \eqref{eq:intro_ball_end}. If
$
\oint_{\bS^2}e^\psi\,d\mu>\frac14,
$
Theorem \ref{thm:intro_line_case} gives a solution with the asymptotic behavior in \eqref{eq:intro_line_end}. In either case, Corollary \ref{cor:uniqueness} shows that any classical solution with the same boundary value satisfying \eqref{eq:intro_degenerate_end} must coincide with this solution. Hence every such solution has the asserted asymptotic behavior, and the existence and uniqueness statements follow as well.

\section{The Poincar\'e--Einstein metrics}\label{sec:pe_metrics}

Let $w$ be a solution of the $SU(\infty)$-Toda equation \eqref{eq:intro_toda}, and set $W=1-\frac12\xi w_\xi$ as in \eqref{eq:intro_W}. Then
\begin{align}\label{eq:pe_W_equation}
	\Delta W+(We^w)_{\xi\xi}=0,\quad W(0,\cdot)=1.
\end{align}

\begin{lemma}\label{lem:pe_W_positive}
	For every solution $w$ given by Theorem \ref{thm:intro_general_case} or Theorem \ref{thm:intro_line_case}, we have $W=1-\frac12\xi w_\xi>0$ on $[0,\frac12)\times\bS^2$.
\end{lemma}

\begin{proof}
	The estimates \eqref{eq:intro_general_limit} and \eqref{eq:intro_line_derivative_bound} give $(-We^w)_+\rightarrow0$ uniformly as $\xi\rightarrow\frac12$. Since $We^w=e^\psi>0$ at $\xi=0$, the positive-part argument in the proof of Lemma \ref{lem:comparison}, applied to \eqref{eq:pe_W_equation} with $-We^w$ and $\Theta=e^{-w}$, gives $W\geq0$ throughout $[0,\frac12)\times\bS^2$.

	For every $0<b<\frac12$, equation \eqref{eq:pe_W_equation} is uniformly elliptic on $[0,b]\times\bS^2$. Since $W(0,\cdot)=1$, the interior Harnack inequality gives $W>0$ on $(0,b)\times\bS^2$. As $b$ is arbitrary, the conclusion follows.
\end{proof}

\subsection{The 4-ball case}

Let $\psi\in C^\infty(\bS^2)$ satisfy $\oint_{\bS^2}e^\psi\,d\mu=\frac14$, and let $w$ be the solution given by Theorem \ref{thm:intro_general_case}.

Let $\pi:\bS^3\rightarrow\bS^2$ be the Hopf fibration, with circle period $2\pi$, and use the same notation for its product with $[0,\frac12)$. We omit the pullback symbol $\pi^*$ when writing functions, differential forms, and metrics from the base. Let $d\mathrm{vol}_{\bS^2}$ denote the area form of $g_{\bS^2}$, with total area $4\pi$, and choose the standard Hopf connection one-form $\eta_0$ so that
\begin{align}\label{eq:ball_hopf_normalization}
	d\eta_0=-\frac12d\mathrm{vol}_{\bS^2},\quad
	g_{\bS^3}=\eta_0^2+\frac14g_{\bS^2}
\end{align}
is the unit round metric. We choose a connection one-form $\eta$ of $\pi:\bS^3\times [0,\frac12)\rightarrow\bS^2\times [0,\frac12)$ satisfying
\begin{align}\label{eq:pe_connection_curvature}
	d\eta=(We^w)_\xi\,d\mathrm{vol}_{\bS^2}+*_{\bS^2}(d_{\bS^2}W)\wedge d\xi.
\end{align}
Such a connection exists because the right-hand side of \eqref{eq:pe_connection_curvature} is a closed 2-form on $\bS^2\times [0,\frac12)$ by \eqref{eq:pe_W_equation} and has the same integral over each spherical slice as $d\eta_0$, as verified below.
Here $*_{\bS^2}$ is the Hodge star on one-forms: in  oriented conformal coordinates $(x,y)$,
$$*_{\bS^2}(d_{\bS^2}W)=W_x\,dy-W_y\,dx.$$
Moreover, \eqref{eq:intro_spherical_mean} gives
\begin{align}\label{eq:ball_connection_period}
	\oint_{\bS^2}We^w\,d\mu=\frac12(\frac12-\xi),\quad
	\oint_{\bS^2}(We^w)_\xi\,d\mu=-\frac12.
\end{align}
Thus its integral of $d\eta$ over each spherical slice is $-2\pi$, agreeing with the curvature of $\eta_0$.

The choice of $\eta$ in \eqref{eq:pe_connection_curvature} is not unique. Any two such connections differ by a closed one-form on the base $\bS^2\times[0,\frac12)$, hence by $df$ for some smooth function $f(\xi,x)$, since the base is simply connected. Locally, we can write $\eta=d\theta+\beta$, where $\theta$ is the circle coordinate and $\beta$ is a one-form on the base. The map $\theta\mapsto\theta+f(\xi,x)$, leaving $(\xi,x)$ fixed, pulls $\eta$ back to $\eta+df$. This is what we mean by a change of circle coordinate, or a gauge transformation. The metrics in \eqref{eq:intro_metric_ansatz} change only by pullback under the same map.

Define $g$ and $h$ by \eqref{eq:intro_metric_ansatz}. By LeBrun's construction \cite[Proposition 1]{LeBrun1991Toda} and Tod's Einstein condition \cite[Section 1]{Tod2006}, or \cite{LiLiu2025}, $g$ is scalar-flat K\"ahler and $h$ is anti-self-dual Einstein, with $\mathrm{Ric}(h)=-3h$, away from $\xi=\frac12$.

\begin{proposition}\label{prop:ball_metric_extension}
	For a suitable choice of $\eta$ satisfying \eqref{eq:pe_connection_curvature}, the metric $g$ extends to a smooth scalar-flat K\"ahler metric on $\overline{B^4}$. The metric $h=\xi^{-2}g$ is a complete anti-self-dual Poincar\'e--Einstein metric on $B^4$ with $\bS^1$ symmetry, and its conformal boundary is represented by $\eta^2|_{\xi=0}+e^\psi g_{\bS^2}$.
\end{proposition}

\begin{proof}
	Boundary elliptic regularity gives smoothness at $\xi=0$. It therefore suffices to find good coordinates around $\xi=\frac12$. Define $s=-\log(1-2\xi)$ and $u(s,x)=(\frac12-\xi)^{-2}e^{w(\xi,x)}$ as in \eqref{eq:ball_transform}. By Theorem \ref{thm:general_existence_decay}, $u$ converges to a conformal Jacobian $\mathfrak{j}_a$. Choose a fixed orientation-preserving conformal transformation $\phi$ such that $J_\phi(\mathfrak{j}_a\circ\phi)=1$, so $u_\phi=J_\phi(u\circ\phi)\rightarrow1$. Let $\Phi:\bS^3\times[0,\frac12)\rightarrow\bS^3\times[0,\frac12)$ be a lifted bundle diffeomorphism, which leaves $\xi$ fixed and sends the Hopf circle over $x$ to the Hopf circle over $\phi(x)$. Since $\phi^*g_{\bS^2}=J_\phi g_{\bS^2}$, we have
	\begin{align}\label{eq:ball_conformal_metric}
		\begin{split}
			\Phi^*g&=(W\circ\phi)d\xi^2+(W\circ\phi)^{-1}(\Phi^*\eta)^2+J_\phi((We^w)\circ\phi)g_{\bS^2}\\
			&=W_\phi d\xi^2+W_\phi^{-1}(\Phi^*\eta)^2+W_\phi e^{w_\phi}g_{\bS^2}.
		\end{split}
	\end{align}
	Here $w_\phi=w\circ\phi+\log J_\phi$ and $W_\phi=1-\frac12\xi\partial_\xi w_\phi=W\circ\phi$. Since $\phi^*d\mathrm{vol}_{\bS^2}=J_\phi d\mathrm{vol}_{\bS^2}$ and the Hodge star on one-forms is conformally invariant, while $\phi$ and $J_\phi$ are independent of $\xi$, we obtain
	\begin{align}\label{eq:ball_conformal_connection}
		d(\Phi^*\eta)=\Phi^*(d\eta)=(W_\phi e^{w_\phi})_\xi\,d\mathrm{vol}_{\bS^2}+*_{\bS^2}(d_{\bS^2}W_\phi)\wedge d\xi.
	\end{align}
	Hence $\Phi^*\eta$ satisfies \eqref{eq:pe_connection_curvature} for the transformed solution, so it suffices to prove smooth extension assuming $u\rightarrow1$.

	Set
	\begin{align}\label{eq:ball_radial_coordinate}
		\xi=\frac{1-r^2}{2},\quad s=-2\log r,\quad
		A=r^2W=1-\xi\frac{u_s}{u}.
	\end{align}
	By \eqref{eq:general_existence_bounds}, with $\delta=1/22$, we have $u-1=O(r^{2\delta})$ and $A-1=O(r^{2\delta})$. Interior estimates for \eqref{eq:s_pde} give the same bounds after any fixed number of $r\partial_r$ and spherical derivatives. The metric becomes
	\begin{align}\label{eq:ball_radial_metric}
		g=A\,dr^2+r^2A^{-1}\eta^2+\frac{r^2}{4}uA\,g_{\bS^2}.
	\end{align}
	Define
	\begin{align}\label{eq:ball_radial_correction}
		a(r,x)=\int_0^r\frac{A(q,x)-1}{q}\,dq.
	\end{align}
	The function $a$ is $O(r^{2\delta})$, and this bound remains valid after any fixed number of applications of $r\partial_r$ and spherical derivatives. 
	
	We claim that a suitable gauge transformation $\eta\mapsto\eta+df$, preserving \eqref{eq:pe_connection_curvature}, puts $\eta$ in the form
	\begin{align}\label{eq:ball_connection_gauge}
		\eta=\eta_0+*_{\bS^2}(d_{\bS^2}a).
	\end{align}

	Let $b(r,x)$ be the coefficient of $dr$ in $\eta$, and set
	$F(r,x)=\int_1^r b(q,x)\,dq.$
	Replacing $\eta$ by $\eta-dF$ makes its coefficient of $dr$ vanish. With this choice, write $\eta=\eta_0+\alpha(r,x)$, where $\alpha(r,\cdot)$ is a one-form on $\bS^2$. Then
	$$d\eta=d\eta_0+d_{\bS^2}\alpha+dr\wedge\partial_r\alpha.$$
	Substituting this into \eqref{eq:pe_connection_curvature} and using $W=A/r^2$ and $d\xi=-r\,dr$, we obtain
	$$\partial_r\alpha=\frac1r*_{\bS^2}(d_{\bS^2}A)=\partial_r\bigl(*_{\bS^2}d_{\bS^2}a\bigr).$$
	Integrating in $r$, we obtain
	$$\alpha(r,x)=*_{\bS^2}(d_{\bS^2}a)+\alpha_0(x),$$
	where $\alpha_0$ is a one-form on $\bS^2$ independent of $r$.

	Restricting \eqref{eq:pe_connection_curvature} to each spherical slice gives
	$$\left(-\frac12+\Delta a\right)d\mathrm{vol}_{\bS^2}+d_{\bS^2}\alpha_0=(We^w)_\xi\,d\mathrm{vol}_{\bS^2},$$
	where we used \eqref{eq:ball_hopf_normalization} and $d_{\bS^2}(*_{\bS^2}d_{\bS^2}a)=(\Delta a)d\mathrm{vol}_{\bS^2}$. Since $We^w=r^2uA/4$ and $\partial_\xi=-r^{-1}\partial_r$, the decay estimates for $u$ and $A$ give $(We^w)_\xi\rightarrow-\frac12$ as $r\rightarrow0$. Also, $\Delta a\rightarrow0$. Taking the limit in the spherical identity therefore gives $d_{\bS^2}\alpha_0=0$, since $\alpha_0$ is independent of $r$. The sphere is simply connected, so $\alpha_0=d_{\bS^2}f$ for some function $f(x)$. Replacing $\eta$ by $\eta-df$ gives \eqref{eq:ball_connection_gauge}.

	Consequently, with $g_0=dr^2+r^2g_{\bS^3}$, we have
	\begin{align}\label{eq:ball_euclidean_leading_term}
		|g-g_0|_{g_0}=O(r^{2\delta}).
	\end{align}

	We now use Einstein regularity to obtain smoothness at the center. In the Cartesian coordinates $Z=r\theta\in\bR^4$, $\theta\in\bS^3$, each partial  derivative is a combination of $\partial_r$ and $r^{-1}$ times angular derivatives on $\bS^3$. The derivative bounds for $u$, $A$, and $a$ therefore give $|\partial_Z(g-g_0)|_{g_0}=O(r^{2\delta-1})$. Moreover, since $h-4g_0=\xi^{-2}(g-g_0)+(\xi^{-2}-4)g_0$ and $\delta<1$, 
	 we obtain
	\begin{align}\label{eq:ball_einstein_remainder}
		|\partial_Z^j(h-4g_0)|_{g_0}\leq C_j|Z|^{2\delta-j},\quad j=0,1.
	\end{align}
	These bounds give a positive definite $W^{1,p}\hookrightarrow C^{0,1-4/p}$ extension of $h$ for some $p>4$, satisfying $\mathrm{Ric}(h)=-3h$ weakly across the center. We apply the harmonic-coordinate regularity argument of DeTurck--Kazdan \cite[Sections 4--5]{DeTurckKazdan1981}. In harmonic coordinates, the Einstein equation takes the form
	\begin{align}\label{eq:harmonic_coordinates}
		-\frac12 h^{ab}\partial_a\partial_b h_{ij}
		+Q_{ij}(h^{-1},\partial h)
		=-3h_{ij},
	\end{align}
	where $Q_{ij}$ is quadratic in the first derivatives of $h$, with coefficients depending smoothly on $h^{-1}$. Since $h\in W^{1,p}$ with $p>4$, we have
	$Q_{ij}(h^{-1},\partial h)\in L^{p/2}.$ Elliptic regularity therefore gives $h\in W^{2,p/2}.$ Since $p/2>2$, Sobolev embedding and repeated application of \eqref{eq:harmonic_coordinates} improve the integrability of $\partial h$ and yield $h\in W^{2,q}$
	for some $q>4$. Hence $h\in C^{1,\alpha'}$ for some $\alpha'>0$, and Schauder estimates bootstrap $h$ to smoothness across the center.

	Moreover, the metric ansatz and \eqref{eq:intro_W} give
	\begin{align}\label{eq:ball_defining_function_equation}
		\Delta_h\xi=\frac{\xi^2w_\xi-2\xi}{W}=-2\xi.
	\end{align}
	Since $\xi$ is bounded, the isolated singularity in this linear equation is removable. Elliptic regularity therefore makes $\xi$ and $g=\xi^2h$ smooth. The circle action extends isometrically, fixing the center. At $\xi=0$ we have $W=1$, so $\xi$ is a boundary defining function, $|d\xi|_g=1$, and the induced boundary metric is $\eta^2|_{\xi=0}+e^\psi g_{\bS^2}$. This proves the stated smooth compactification and completeness of $h$.
	
	Finally, since $g$ is K\"ahler on the punctured ball, and smooth on the ball, its parallel complex structure extends smoothly across the center. Hence the extension of $g$ is K\"ahler.
\end{proof}

\subsection{The complex line bundle over $\bS^2$ case}

Smooth extension of the metrics $g,h$ across the zero section was treated in \cite[Section 4.4]{LiLiu2025}. For completeness, we include a proof here. Let $\mathcal{O}(-m)$ be the complex line bundle over $\bS^2$ of degree $-m$, where $m\geq3$ is an integer. Let $P_m$ be its circle bundle and $\overline{D_m}$ the associated closed disk bundle, with interior $D_m$.

Let $\psi\in C^\infty(\bS^2)$ satisfy
\begin{align}\label{eq:line_metric_boundary_data}
	\oint_{\bS^2}e^\psi\,d\mu=\frac{m-1}{4}>\frac14
\end{align}
and let $w$ be the solution given by Theorem \ref{thm:intro_line_case}. On $P_m\times[0,\frac12)$, use circle period $\pi$ and choose $\eta$ satisfying \eqref{eq:pe_connection_curvature}. As in the ball case, existence follows from closedness of the right-hand side and the curvature integral. Indeed, \eqref{eq:intro_spherical_mean} and \eqref{eq:intro_W} give
\begin{align}\label{eq:line_metric_connection_period}
	\oint_{\bS^2}(We^w)_\xi\,d\mu=-\oint_{\bS^2}e^\psi\,d\mu-\frac14=-\frac m4.
\end{align}
Thus the right-hand side of \eqref{eq:pe_connection_curvature} has integral $-m\pi$ over each spherical slice, as required for degree $-m$.

By Lemma \ref{lem:pe_W_positive}, $W>0$. Define $g$ and $h$ by \eqref{eq:intro_metric_ansatz}; the same construction as in the ball case gives a scalar-flat K\"ahler metric $g$ and an anti-self-dual Einstein metric $h$ away from $\xi=\frac12$.

\begin{proposition}\label{prop:line_bundle_metric_extension}
	For a suitable choice of $\eta$ satisfying \eqref{eq:pe_connection_curvature}, the metric $g$ extends to a smooth scalar-flat K\"ahler metric on $\overline{D_m}$. The metric $h=\xi^{-2}g$ is a complete anti-self-dual Poincar\'e--Einstein metric with $\bS^1$ symmetry on $D_m$, which is diffeomorphic to $\mathcal{O}(-m)$. The conformal boundary is represented by $\eta^2|_{\xi=0}+e^\psi g_{\bS^2}$ on $P_m$.
\end{proposition}

\begin{proof}
	As in Proposition \ref{prop:ball_metric_extension}, boundary elliptic regularity gives smoothness at $\xi=0$, so it remains to extend the metric across $\xi=\frac12$. Write $e^w=(\frac12-\xi)u$ as in \eqref{eq:line_bundle_transform}. By Theorem \ref{thm:intro_line_case}, $u$ extends smoothly and positively to $\xi=\frac12$.

	Set
	\begin{align}\label{eq:line_metric_radial_coordinate}
		\xi=\frac12-r^2,\quad A=4r^2W=1+2r^2-2\xi r^2\frac{u_\xi}{u}.
	\end{align}
	Then $u$ and $A$ are smooth functions of $(r^2,x)$, with $A(0,x)=1$, and the metric becomes
	\begin{align}\label{eq:line_metric_radial_form}
		g=A\,dr^2+r^2A^{-1}(2\eta)^2+\frac14uA\,g_{\bS^2}.
	\end{align}
	Since $W-1/(4r^2)$ and $We^w=uA/4$ are smooth in $(r^2,x)$, the curvature in \eqref{eq:pe_connection_curvature} is smooth up to $r=0$. Choosing a gauge so that $2\eta$ has no $dr$ component, we may therefore write locally
	\begin{align}\label{eq:line_metric_connection_gauge}
		2\eta=d\theta+\alpha(r^2,x),
	\end{align}
	where $\theta$ has period $2\pi$ and $\alpha(r^2,\cdot)$ is a smooth family of one-forms on the base.

	Use the Cartesian fiber coordinates $z=y_1+iy_2=re^{i\theta}$. Since $(A-1)/r^2$ is smooth, the identities
	$$dr^2+r^2d\theta^2=dy_1^2+dy_2^2,\quad r\,dr=y_1dy_1+y_2dy_2,\quad r^2d\theta=y_1dy_2-y_2dy_1$$
	show that \eqref{eq:line_metric_radial_form} extends smoothly across $y_1=y_2=0$. At the zero section, its value is
	$$dy_1^2+dy_2^2+\frac14u(\frac12,x)g_{\bS^2},$$
	which is positive definite. The circle action becomes rotation in the fibers and fixes the zero section. Moreover, $\xi=\frac12-|z|^2$ is smooth and positive there, so $h=\xi^{-2}g$ also extends smoothly.

	Finally, at $\xi=0$ we have $W=1$, so $|d\xi|_g=1$ and the induced boundary metric is $\eta^2|_{\xi=0}+e^\psi g_{\bS^2}$. As in Proposition \ref{prop:ball_metric_extension}, this proves the stated smooth compactification and completeness of $h$. Since $g$ is K\"ahler away from the smooth codimension-$2$ zero section and extends smoothly across it, its parallel complex structure $J$ satisfies $|\partial J|\leq C$ in smooth local coordinates. The codimension-$2$ condition gives a unique continuous extension of $J$, and $\nabla^gJ=0$ then implies that this extension is smooth. Hence the extension of $g$ is K\"ahler.
\end{proof}

\Addresses

\end{document}